\documentclass[12pt]{amsart}
\usepackage[all]{xy}
\usepackage[latin1]{inputenc}
\usepackage{hyperref}
\usepackage{amssymb}
\newcommand{\C}{{\mathbb C} }

\newcommand{\cA}{{\mathcal A} }

\newcommand{\cE}{{\mathcal E} }

\newcommand{\cL}{{\mathcal L} }
\newcommand{\cM}{{\mathcal M} }

\newcommand{\cO}{{\mathcal O} }

\newcommand{\cT}{{\mathcal T} }

\newcommand{\cX}{{\mathcal X} }

\newcommand{\cK}{{\mathcal K} }
\newcommand{\wh}{\widehat}
\newcommand{\wt}{\widetilde}
\newcommand{\pt}{\partial}
\def\ol#1{{\overline{#1}}}
\newtheorem*{theorem*}{Theorem}
\newtheorem{maintheorem}{Main Theorem}
\newtheorem*{Maintheorem*}{Main Theorem}
\newtheorem{theorem}{Theorem}
\newtheorem{definition}{Definition}
\newtheorem{lemma}{Lemma}

\newtheorem{proposition}{Proposition}
\newtheorem{corollary}{Corollary}

\def\ke{K{\"a}h\-ler-Ein\-stein }
\def\ks{Ko\-dai\-ra-Spen\-cer }
\def\ka{K{\"a}h\-ler }

\def\wp{Weil-Pe\-ters\-son }

\def\tei{Teich\-mül\-ler }

\def\ii{\sqrt{-1}}
\def\ddb{\sqrt{-1}\partial\overline{\partial}}
\def\C{\mathbb{C}}

\def\cinf{C^\infty}

\def\gab{{g_{\alpha\ol\beta}}}

\def\RP{$R^{n-p}f_*\Omega^p_{\cX/S}(\cK_{\cX/S}^{\otimes m})$\ }
\def\we{\wedge}

\begin{document}

\title[Curvature of $R^{n-p}f_*\Omega^p_{\cX/S}(\cK_{\cX/S}^{\otimes m})$
 and applications]{Curvature of $R^{n-p}f_*\Omega^p_{\cX/S}(\cK_{\cX/S}^{\otimes m})$\\ and applications}

\author{Georg Schumacher}
\address{Fachbereich Mathematik und Informatik,
Philipps-Universit\"at Marburg, Lahnberge, Hans-Meerwein-Straße, D-35032
Marburg,Germany}
\email{schumac@mathematik.uni-marburg.de}
\date{}
\maketitle
\begin{abstract}
Given an effectively parameterized family $f:\cX\to S$ of canonically
polarized manifolds, the \ke metrics on the fibers induce a hermitian
metric on the relative canonical bundle $\cK_{\cX/S}$. We use a global
elliptic equation to show that this metric is strictly positive everywhere
and give estimates.

The direct images $R^{n-p}f_*\Omega^p_{\cX/S}(\cK_{\cX/S}^{\otimes m})$,
$m > 0$, carry induced natural hermitian metrics. We prove an explicit
formula for the curvature tensor of these direct images. The formula for
$p=n$, implies that $f_*(\cK_{\cX/S}^{\otimes (m+1)})$ is Nakano-positive
with estimates (for effectively parameterized families). We apply it to
the morphisms $S^p\cT_S \to R^pf_*\Lambda^p\cT_{\cX/S}$ induced by the \ks
map and obtain a differential geometric proof for hyperbolicity properties
of $\cM_{\text{can}}$. Similar results hold for families of polarized
Ricci-flat manifolds. These will appear elsewhere.\footnote{}
\end{abstract}
\footnotetext{Dissertation in progress.}

\section{Introduction}
For any holomorphic family $f: \cX \to S$ of canonically polarized,
complex manifolds, the unique \ke metrics on the fibers define an
intrinsic metric on the relative canonical bundle $\cK_{\cX/S}$. The
construction is functorial in the sense of compatibility with base
changes. By definition, its curvature form has at least as many positive
eigenvalues as the dimension of the fibers indicates. It turned
\cite{sch-preprint} out that it is {\em strictly positive}, provided the
induced deformation is not infinitesimally trivial at the corresponding
point of the base.

Actually the first variation of the metric tensor in a family of compact
\ke manifolds contains the information about the induced deformation, more
precisely, it contains the harmonic representatives $A_s=
A^\alpha_{s\ol\beta}\pt_\alpha dz^\ol\beta $ of the \ks classes
$\rho(\pt/\pt s)$. The positivity of the hermitian metric will be measured
in terms of a certain global function. Essential is an elliptic equation
on the fibers, which relates this function to the pointwise norm of the
harmonic \ks forms. The ''strict'' positivity of the corresponding
(fiberwise) operator $(\Box + id)^{-1}$, where $\Box$ is the complex
Laplacian, can be seen in a direct way (cf. \cite{sch-preprint}).  For
families of compact Riemann surfaces the elliptic equation was previously
derived in terms of automorphic forms by Wolpert \cite{wo}. Later in
higher dimensions a similar equation arose in the work of Siu
\cite{siu:canlift} for families of canonical polarized manifolds.

In this article, we will reduce estimates for the positivity of the
curvature of $\cK_{\cX/S}$ on $\cX$ to estimates of the resolvent kernel
of the above integral operator, whose its positivity was already shown by
Yosida in \cite{yos}. Finally estimates for the resolvent kernel follow
from the estimates for the heat kernel, which were achieved by Cheeger and
Yau in \cite{c-y}.

The positivity of the relative canonical bundle is important, when
estimating the curvature of the direct image sheaves
$R^{n-p}f_*\Omega^p_{\cX/S}(\cK_{\cX/S}^{\otimes m})$. These are equipped
with a natural hermitian metric that is induced by the $L^2$-inner product
of harmonic tensors on the fibers of $f$. We will give an explicit formula
for the curvature tensor.

A main motivation of our approach is the study of the curvature of the
classical \wp metric, in particular the computation of the curvature
tensor by Tromba \cite{tr} and Wolpert \cite{wo}. It immediately implies
the hyperbolicity of the classical \tei space.  The curvature tensor of
the generalized \wp metric for families of metrics with constant negative
Ricci curvature was explicitly  computed by Siu in \cite{siu:canlift}; in
\cite{sch:curv} a formula in terms of elliptic operators and harmonic \ks
tensors was derived. However, the curvature of the generalized \wp metric
seems not to satisfy any negativity condition. This difficulty was
overcome in the work of Viehweg and Zuo in \cite{v-z}. Their approach to
hyperbolicity makes use of the period map in the sense of Griffiths.

On the other hand our results are motivated by Berndtsson's result on the
Nakano-positivity for certain direct images. In \cite{sch-preprint} we
showed that the positivity of $f_*\cK_{\cX/S}^{\otimes 2}$ together with
the curvature of the generalized \wp metric is sufficient to imply a
hyperbolicity result for moduli of canonically polarized complex manifolds
in dimension two.

For ample $K_X$ the cohomology groups $H^{n-p}(X,\Omega^p_X(K^{\otimes
m}_X))$ are critical with respect to the Kodaira-Nakano vanishing theorem.
The understanding of this situation is the other main motivation. We will
consider the relative case. Let $A_{i \ol\beta}^\alpha(z,s) \pt_\alpha
dz^{\ol\beta}$ be a harmonic \ks form. Then for  $s \in S$ the cup product
together with the contraction defines
\begin{eqnarray*}
A_{i \ol\beta}^\alpha \pt_\alpha dz^{\ol\beta}\cup \textvisiblespace :
\cA^{0,n-p}(\cX_s,\Omega^p_{\cX_s}(\cK_{\cX_s}^{\otimes m})) &\to&
\cA^{0,n-p+1}(\cX_s,\Omega^{p-1}_{\cX_s}(\cK_{\cX_s}^{\otimes m}))\\ A_{\ol \jmath
\alpha}^\ol\beta \pt_\ol\beta dz^{\alpha}\cup \textvisiblespace :
\cA^{0,n-p}(\cX_s,\Omega^p_{\cX_s}(\cK_{\cX_s}^{\otimes m})) &\to&
\cA^{0,n-p-1}(\cX_s,\Omega^{p+1}_{\cX_s}(\cK_{\cX_s}^{\otimes m})).
\end{eqnarray*}
\enlargethispage{.5cm} We will apply the above product to harmonic
$(0,n-p)$-forms. In general the result is not harmonic. We denote the
pointwise $L^2$ inner product by a dot.
\begin{maintheorem}
The curvature tensor for $R^{n-p}f_*\Omega^p_{\cX/S}(\cK_{\cX/S}^{\otimes
m})$ is given by
\begin{eqnarray*}
R_{i\ol\jmath}^{\phantom{{i\ol\jmath}}\ol\ell k}(s)&=& m \int_{\cX_s}
\left( \Box + 1 \right)^{-1}(A_i\cdot A_\ol\jmath) \cdot(\psi^k \cdot
\psi^\ol\ell) g\/ dV\nonumber\\
&& \quad + m \int_{\cX_s} \left( \Box + m \right)^{-1} (A_i\cup\psi^k)
\cdot (A_\ol\jmath \cup \psi^\ol\ell) g\/ dV \\
&& \quad + m \int_{\cX_s} \left( \Box - m \right)^{-1}
(A_i\cup\psi^\ol\ell)\cdot (A_\ol\jmath \cup \psi^k) g\/ dV.
\nonumber
\end{eqnarray*}
\end{maintheorem}
The potentially negative third term is not present for $p=n$, i.e.\ for
$f_*\cK_{\cX/S}^{\otimes(m+1)}$. From the main theorem we get immediately
a fact which is contained in Berndtsson's theorem \cite{berndtsson}:

{\em The locally free sheaf $f_*\cK_{\cX/S}^{\otimes(m+1)}$ is
Nakano-positive.}

(Very recently Berndtsson considered the curvature of $f_* (\cK_{\cX/S}
\otimes \cL)$, \cite{berndtsson-pre}.)

We prove the estimate
$$
R(A,\ol A, \psi,\ol\psi) \geq P_n(d(\cX_s)) \cdot \|A\|^2\cdot
\|\psi\|^2,
$$
where the coefficient $P_n(d(\cX_s))>0$ is an explicit function depending
on the dimension and the diameter of the fibers.

By Serre duality the theorem yields the following version, which contains
for $p=1$ the curvature formula for the generalized \wp metric. Again a
tangent vector of the base is identified with a harmonic \ks form $A_i$,
and $\nu_k$ stands for a section of the relevant sheaf:
\begin{maintheorem}
The curvature for $R^pf_*\Lambda^p\cT_{\cX/S}$ equals
\begin{eqnarray*}
R_{i\ol\jmath   k \ol\ell }(s)&=&- \int_{\cX_s}
\left( \Box + 1 \right)^{-1}(A_i\cdot A_\ol\jmath)
\cdot(\nu_k \cdot \nu_\ol\ell) g\/ dV\nonumber\\
&& \quad - \int_{\cX_s} \left( \Box + 1 \right)^{-1} (A_i\wedge\nu_\ol\ell)
\cdot (A_\ol\jmath \wedge \nu_k) g\/ dV \\
&& \quad -  \int_{\cX_s} \left( \Box - 1 \right)^{-1}
(A_i\wedge \nu_k)\cdot (A_\ol\jmath \wedge \nu_\ol\ell) g\/ dV.
\nonumber
\end{eqnarray*}
\end{maintheorem}

In order to prove a  result about hyperbolicity of moduli spaces we use
Demailly's approach. An upper semi-continuous Finsler metric of negative
holomorphic curvature on a relatively compact subspace of the moduli stack
of canonically polarized varieties can be constructed so that any such
space is hyperbolic with respect to the orbifold structure.

We get immediately the following fact related to Shafarevich's
hyperbolicity conjecture for higher dimensions solved by Migliorini
\cite{m}, Kovács \cite{kv1,kv2,kv3}, Bedulev-Viehweg \cite{b-v}, and
Viehweg-Zuo \cite{v-z,v-z2}.

{\bf Application.} {\it Let $\cX \to C$ be a non-isotrivial holomorphic
family of canonically polarized manifolds over a curve. Then $g(C)>1$.}

{\it Acknowledgements.} This work was begun during a visit to Harvard
University. The author would like to thank Professor Yum-Tong Siu for his
cordial hospitality and many discussions. His thanks also go to Bo
Berndtsson, Jeff Cheeger, Jean-Pierre Demailly, Gordon Heier, Stefan
Kebekus, Janos Kollár, Sándor Kovács, and Thomas Peternell for discussions
and remarks.
\bigskip

\bigskip

\section{Fiber integration and Lie derivatives}

\subsection{Definition of fiber integrals and basic properties}
Let denote $\{\cX_s\}_{s\in S}$ a holomorphic family of compact complex
manifolds $\cX_s$ of dimension $n>0$ parameterized by a reduced complex
space $S$. By definition, it is given by a proper holomorphic submersion
$f:\cX \to S$, such that the $\cX_s = f^{-1}(s)$ for $s\in S$. In case of
a smooth space $S$, if $\eta$ is a differential form of class $\cinf$ of
degree $2n+r$ the fiber integral
$$
\int_{\cX/S} \eta
$$
is a differential form of degree $r$ on $S$. It can be defined as follows:
Fix a point $s_0\in S$ and denote by $X=\cX_{s_0}$ the fiber. Let $U
\subset S$ be an open neighborhood of $s_0$  such that there exists a
$\cinf$ trivialization of the family:
$$
\xymatrix{f^{-1}U  \ar[d]_{f|f^{-1}U } & X \ar[l]_{\Phi}^\sim\ar[dl]^{pr}\times U\\ U }
$$
Let $z=(z^1,\ldots,z^n)$ and $s=(s^1,\ldots,s^k)$ denote local
(holomorphic) coordinates on $X$ and $S$ resp.

The pull-back $\Phi^*\eta$ possesses a summand $\eta'$, which is of the
form $\sum \eta_k(z,s) dV_z \wedge d\sigma^{k_1}\wedge\ldots\wedge
d\sigma^{k_r}$, where the $\sigma^\kappa$ run through the real and complex
parts of $s^j$, and where $dV_z$ denotes the relative Euclidean volume
element. Now
$$
\int_{\cX/S} \eta := \int_{X\times S/S} \Phi^*\eta := \sum_{k=(k_1,\ldots,k_r)} \left(\int_{\cX_s}
\eta_k(z,s) dV_z \right) d\sigma^{k_1}\wedge\ldots\wedge d\sigma^{k_r}.
$$
The definition is independent of the choice of coordinates and
differentiable trivializations. The fiber integral coincides with the
push-forward of the corresponding current. Hence, if  $\eta$ is a
differentiable form of type $(n+r,n+s)$, then the fiber integral is of
type $(r,s)$.

Singular base spaces are treated as follows: Using deformation theory, we
can assume that $S\subset W$ is a closed subspace of some open set $W
\subset \mathbb C^N$, and that an almost complex structure is defined on
$X\times S$ so that $\cX$ is the integrable locus. Then by definition, a
differential form of class $\cinf$ on $\cX$ will be given on the whole
ambient space $X \times W$ (with a type decomposition defined on $\cX$).

Fiber integration commutes with taking exterior derivatives:
\begin{equation}\label{eq:fibintallg}
d \int_{\cX/S} \eta = \int_{\cX/S} d\eta,
\end{equation}
and since it preserves the type (or to be seen explicitly in local
holomorphic coordinates), the same equation holds true for $\pt$ and
$\ol\pt$ instead of $d$.

A \ka form $\omega_\cX$ on a singular space, by definition is a form that
possesses locally a $\pt\ol\pt$-potential, which is the restriction of a
$\cinf$ function on a smooth ambient space. It follows from the above
facts that given a \ka form  $\omega_\cX$  on the total space, the fiber
integral
\begin{equation}
\int_{\cX/S} \omega_\cX^{n+1}
\end{equation}
is a \ka form on the base space $S$, which possesses locally a smooth
$\pt\ol\pt$-potential, even if the base space of the smooth family is
singular.

For the actual computation of exterior derivatives of fiber integrals
\eqref{eq:fibintallg}, in particular of functions, given by integrals of
$(n,n)$-forms on the fibers, the above definition seems to be less
suitable. Instead the problem is reduced to derivatives of the form
\begin{equation}\label{eq:derfibint}
\frac{\pt}{\pt s^i} \int_{\cX_s} \eta,
\end{equation}
where only the vertical components of $\eta$ contribute to the integral.
Here and later we will always use the summation convention.
\begin{lemma}\label{le:intLie}
Let
$$
w_i=\left.\left(\frac{\pt}{\pt s^i} + b^\alpha_i(z,s)\frac{\pt}{\pt z^\alpha}
+ c^\ol\beta(z,s)\frac{\pt}{\pt z^\ol\beta} \right)\right|_s
$$
be differentiable vector fields, whose projection to $S$ equal
$\frac{\pt}{\pt s^i}$. Then
$$
\frac{\pt}{\pt s^i} \int_{\cX_s} \eta =
\int_{\cX_s} L_{w_i}\left(\eta\right),
$$
where $L_{w_i}$ denotes the Lie derivative.
\end{lemma}

Concerning singular base spaces, observe that it is sufficient that the
above equation is given on the first infinitesimal neighborhood of $s$ in
$S$.

\begin{proof}
Because of linearity, one may consider the real and imaginary parts of
$\pt/\pt s^i$ and $w_i$ resp.\ separately.

Let $\pt/\pt t$ stand for $Re(\pt/\pt s^i)$, and let $\Phi_t: X \to \cX_t$
be the one parameter family of diffeomorphisms generated by $Re(w_i)$.
Then
$$
\frac{d}{d t} \int_{\cX_s} \eta = \int_X \frac{d}{dt} \Phi^*_t \eta = \int_X L_{Re(w^i)}(\eta).
$$
It is known that the vector fields $Re(w_i)$ and $Im(w_i)$ need not
commute.
\end{proof}
In our applications, the form $\eta$ will typically consist of inner
products of differential forms with values in hermitian vector bundles,
whose factors need to be treated separately. This will be achieved by the
Lie derivatives. In this context, we will have to use covariant
derivatives with respect to the given hermitian vector bundle on the total
space and to the \ka metrics on the fibers.

\subsection{Direct images and differential forms}\label{ss:dolb}
Let $(\cE, h)$ be a hermitian, holomorphic vector bundle on $\cX$, whose
direct image $R^q f_*\cE$ is {\em locally free}. Furthermore we assume
that for all $s\in S$ the cohomology $H^{q+1}(\cX_s, \cE \otimes
\cO_{\cX_s})$ {\em vanishes}. Then the statement of the
Grothendieck-Grauert comparison theorem holds for $R^q f_*\cE$, in
particular $R^q f_*\cE\otimes_{\cO_S} \C(s)$ can be identified with
$H^{q}(\cX_s, \cE \otimes_{\cO_\cX}\cO_{\cX_s})$.

For simplicity, we assume that he base space $S$ is smooth. Locally, after
replacing $S$ by a neighborhood of any given point, we can represent
sections of the $q$-th direct image sheaf in terms of Dolbeault cohomology
by $\ol\pt$-closed $(0,q)$-forms. On the other hand, fiberwise, we have
harmonic representatives of cohomology classes with respect to the \ka
form and hermitian metric on the fibers. The following fact will be
essential.

\begin{lemma}\label{le:dolb}
Let $\wt\psi \in R^q f_*\cE(S)$ be a section. Let $\psi_s \in
\cA^{0,q}(\cX_s,\cE_s)$ be the harmonic representatives of the cohomology
classes $\wt\psi|\cX_s$.

Then locally with respect to $S$ there exists a $\ol\pt$-closed form
$\psi\in \cA^{0,q}(\cX,\cE)$, which represents $\wt \psi$, and whose
restrictions to the fibers $\cX_s$ equal $\psi_s$.
\end{lemma}

We omit the simple proof.


In this way, the relative Serre duality can be treated in terms of such
differential forms. Let $\cE^\vee = \textit{Hom}_\cX(\cE,\cO_\cX)$. Then
$$
R^pf_*\cE \otimes_{\cO_S} R^{n-p}f_*(\cE^\vee \otimes_{\cO_\cX} \cK_{\cX/S}) \to \cO_S
$$
is given by the fiber integral of the wedge product of $\ol\pt$-closed
differential forms in the sense of Lemma~\ref{le:dolb}. By
\eqref{eq:fibintallg} (for the operator $\ol\pt$), the result is a
$\ol\pt$-closed $0$-form i.e.\ holomorphic function.

\section{Estimates for resolvent and heat kernel}
Let $(X,\omega_X)$ be a compact \ka manifold.
The Laplace operator for differentiable functions is given by $\Box =
\ol\pt\ol\pt^*+\ol\pt^*\ol\pt$, where the adjoint $\ol\pt^*$ is the formal
adjoint operator. The Laplacian is self-adjoint with non-negative
eigenvalues.

The corresponding resolvent operator $(id + \Box)^{-1}$ is defined on the
space of continuous functions and bounded.

First, we observe that the resolvent operator is positive: If $\chi \geq
0$ everywhere on $X$, then the function given by $(\Box + id)^{-1}\chi$ is
non-negative. This fact follows immediately from the minimum principle
applied to the elliptic equation
$$
\Box \phi + \phi =  \chi.
$$
So the integral kernel $P(z,w)$ must be non-negative for all $z$ and $w$.

For any function $\chi(z)$
$$
(\Box + id)^{-1}(\chi)(z)= \int_X P(z,w)\chi(w) g(w)dV_w
$$
holds. In a similar way we denote by $P(t,z,w)$ the integral kernel for
the heat operator
$$
\frac{d}{dt} + \Box
$$
so that the solution of the heat equation with initial function $\chi(z)$
for $t=0$ is given by
$$
\int_X P(t,z,w)\chi(w) g(w)dV_w.
$$
The explicit representation of the above operators in terms of eigen
functions of the Laplacian yields the following relation.

\begin{lemma}
\label{le:heat} Let $P(z,w)$ be the integral kernel of the resolvent
operator and denote by $P(t,z,w)$ the heat kernel. Then
\begin{equation}\label{eq:resheatest}
P(z,w)= \int_0^\infty e^{-t}P(t,z,w)dt.
\end{equation}
\end{lemma}

\begin{proof}
Let $\{\chi_\nu\}$ be a set of eigenfunctions of the Laplacian with
eigenvalues $\lambda_\nu$ so that
$$
P(z,w)= \sum_\nu \frac{1}{1+\lambda_\nu} \chi_\nu(z)\chi_\nu(w)
$$
and
$$
P(t,z,w)= \sum_\nu e^{-t\lambda_\nu} \chi_\nu(z)\chi_\nu(w).
$$
Then, since the eigenvalues are non-negative,
$$
\int_0^\infty e^{-t(\lambda + 1)} dt = \frac{1}{1+\lambda}
$$
implies \eqref{eq:resheatest}, (cf.\ also \cite[(3.13)]{c-y}).
\end{proof}

We now apply the lower estimates for the heat kernel by Cheeger and Yau
\cite{c-y} to the resolvent kernel. Assuming constant negative Ricci
curvature $-1$, we use the estimates from \cite[(4.3) Corollary]{st}.
\begin{equation}\label{eq:hker}
P(t,z,w)\geq Q_{n}(t,r(z,w)):=\frac{1}{(2\pi t)^n} e^{- \frac{r^2(z,w)}{t}} e^{-\frac{2n-1}{4}t},
\end{equation}
Where $r=r(z,w)$ denotes the geodesic distance (and $n=\dim X$).

Let
\begin{equation}\label{eq:hker1}
P_{n}(r)= \int_0^\infty e^{-t} Q_n(t,r) dt>0.
\end{equation}

Using Lemma~\ref{le:heat} and \eqref{eq:hker} we get
\begin{equation}\label{eq:hker2}
P(z,w) \geq P_n(r(z,w)) \geq P_n(d(X)),
\end{equation}
where $d(X)$ denotes the diameter of $X$. However, $\lim_{r\to
\infty}P_n(r)=0$.

\begin{proposition}\label{pr:resol}
Let $(X,\omega_X)$ be a \ke manifold of constant Ricci curvature $-1$ with
volume element $g\/ dV$ and diameter $d(X)$. Let $\chi$ be a non-negative
continuous function. Let
\begin{equation}\label{eq:ellip}
(1+\Box)\phi = \chi.
\end{equation}
Then
\begin{equation}
  \phi(z)\geq P_n(d(X))\cdot \int_X \chi g\/ dV
\end{equation}
for all $z\in X$.
\end{proposition}
Conversely let for all solutions of \eqref{eq:ellip} an estimate
$\phi(z)\geq P \cdot \int_X \chi g\/ dV$ hold. for some number $P$. Then
$P\leq\inf P(z,w)$ follows immediately.

\medskip

\begin{tiny}
We mention that symbolic integration of \eqref{eq:hker} with
\eqref{eq:hker1} yields an explicit estimate.
$$
P_{n}(r) \geq \frac{1}{(2\pi)^n} \frac{(2n+3)^{\frac{n-1}{2}}}{2^{n-2}} \frac{1}{r^{n-1}} \text{BesselK}\left( n-1, \sqrt{2n+3} r   \right)
$$
\end{tiny}

\section{Positivity of $K_{\cX/S}$}\label{se:posi}
Let $X$ be a canonically polarized manifold of dimension $n$, equipped
with a \ke metric $\omega_X$. In terms of local holomorphic coordinates
$(z^1,\ldots, z^n)$ we write
$$
\omega_X=\ii g_{\alpha\ol\beta}(z)\; dz^\alpha\wedge dz^\ol\beta
$$
so that the \ke equation reads
\begin{equation}\label{eq:ke}
\omega_X=-{\rm Ric}(\omega_X),  \text{ i.e. }  \omega_X= \ddb \log g(z),
\end{equation}
where $g:=\det g_{\alpha\ol\beta}$. We consider $g$ as a hermitian metric
on the anti-canonical bundle $K_X^{-1}$.

For any holomorphic family of compact, canonically polarized manifolds $f:
\cX \to S$ of dimension $n$ with fibers $\cX_s$ for $s\in S$ the \ke forms
$\omega_{\cX_s}$ depend differentiably on the parameter $s$. The resulting
relative \ka form will be denoted by
$$
\omega_{\cX/S} = \ii g_{\alpha,\ol\beta}(z,s)\;dz^\alpha\wedge dz^\ol\beta.
$$
The corresponding hermitian metric on the relative anti-canonical bundle
is given by $g=\det \gab(z,s)$.  We consider the real $(1,1)$-form
$$
\omega_\cX= \ddb \log g(z,s)
$$
on the total space $\cX$. We will discuss the question, whether
$\omega_\cX$ is a \ka form on the total space.

The \ke equation \eqref{eq:ke} implies that
$$
\omega_\cX|\cX_s = \omega_{\cX_s}
$$
for all $s\in S$. In particular $\omega_\cX$, restricted to any fiber, is
positive definite. Our  result is the following statement (cf.\ Main
Theorem).

\begin{theorem}\label{th:main}
Let $\cX \to S$ be a holomorphic family of canonically polarized, compact,
complex manifolds. Then the hermitian metric on $\cK_{\cX/S}$ induced by
the \ke metrics on the fibers is semi-positive and strictly positive on
all fibers. It is strictly positive in horizontal directions, for which
the family is not infinitesimally trivial.
\end{theorem}

Both the statement of the Theorem and the methods are valid for smooth,
proper families of singular (even non-reduced) complex spaces (for the
necessary theory cf.\ \cite{f-s:extremal}).

It is sufficient to prove the theorem for base spaces of dimension one
assuming $S\subset \C$. (In order to treat singular base spaces, the claim
can be reduced to the case where the base is a double point $(0,\mathbb
C[s]/(s^2))$. The arguments below will still be meaningful and can be
applied literally.)

We denote the \ks map for the family $f:\cX \to S$ at a given point
$s_0\in S$ by
$$
\rho_{s_0} :T_{s_0} \to H^1(X, \cT_X)
$$
where $X=\cX_{s_0}$. The family is called {\it effectively parameterized}
at $s_0$, if $\rho_{s_0}$ is injective. The \ks map is induced as edge
homomorphism by the short exact sequence
$$
0 \to  \cT_{\cX/S} \to \cT_\cX \to f^*\cT_S \to 0.
$$
If $v \in T_{s_0}S$ is a tangent vector, say $v=\frac{\pt}{\pt s}|_{s_0}$
and $\frac{\pt}{\pt s} + b^\alpha \frac{\pt}{\pt z^\alpha}$ is any lift of
class $\cinf$ to $\cX$ along $X$, then
$$
\ol\pt\left(\frac{\pt}{\pt s} + b^\alpha(z) \frac{\pt}{\pt z^\alpha}\right)=
\frac{\partial b^\alpha(z)}{\partial z^\ol\beta}
\frac{\pt}{\pt z^\alpha} dz^\ol\beta
$$
is a $\ol\pt$-closed form on $X$, which represents $\rho_{s_0}(\pt / \pt
s)$. Observe that $b^\alpha$ cannot be a tensor on $X$, unless the family
is infinitesimally trivial.

We will use the semi-colon notation as well as raising and lowering of
indices for covariant derivatives with respect to the {\it \ke metrics on
the fibers}. The $s$-direction will be indicated by the index $s$. In this
sense the coefficients of $\omega_\cX$ will be denoted by $g_{s\ol s}$,
$g_{\alpha\ol s}$, $\gab$ etc.

Next, we define {\it canonical lifts} of tangent vectors of $S$ as
differentiable vector fields on $\cX$ along the fibers of $f$ in the sense
of Siu \cite{siu:canlift}. By definition these satisfy  the property that
the induced representative of the \ks class is {\it harmonic}.

Since the form $\omega_\cX$ is positive, when restricted to fibers, {\em
horizontal lifts} of tangent vectors with respect to the pointwise
sesquilinear form $\langle-,-\rangle_{\omega_\cX}$ are well-defined  (cf.\
also \cite{sch:curv}).
\begin{lemma}\label{le:canlift}
The horizontal lift of $\pt/\pt s$  equals
$$
v = \pt_s + a_s^\alpha \pt_\alpha,
$$
where
$$
a_s^\alpha = - g^{\ol\beta \alpha} g_{s \ol \beta}.
$$
\end{lemma}

\begin{proposition}\label{pr:harmrep}
The horizontal lift induces the harmonic representative of
$\rho_{s_0}(\pt/\pt s)$.
\end{proposition}

\begin{proof}
The \ks form of the tangent vector $\pt/\pt_{s_0}$is given by $\ol\pt
v|\cX_s = a^\alpha_{s;\ol\beta}\pt_\alpha dz^\ol\beta$. We consider the
tensor
$$
A^\alpha_{s\ol\beta}:= a^\alpha_{s;\ol\beta}|{\cX_{s_0}}
$$
on $X$. Then
\begin{gather*}
g^{\ol\beta\gamma} A^\alpha_{s\ol\beta;\gamma}= - g^{\ol\beta\gamma}
g^{\ol\delta \alpha} g_{s\ol\delta;\ol\beta\gamma} = - g^{\ol\beta\gamma}
g^{\ol\delta \alpha} g_{s\ol\beta;\ol\delta\gamma} =
-g^{\ol\beta\gamma}g^{\ol\delta \alpha} \left(
g_{s\ol\beta;\gamma\ol\delta} -
g_{s\ol\tau}R^\ol\tau_{\; \ol\beta\ol\delta\gamma} \right)\\
=-g^{\ol\delta\alpha}\left(\left({\pt\log g}/{\pt s} \right)_{;\ol\delta}
+ g_{s\ol \tau}R^\ol\tau_{\; \ol\delta}\right) = 0.
\end{gather*}
\end{proof}
It follows immediately from the proposition that the harmonic \ks forms
induce symmetric tensors. This fact reflects the close relationship
between the \ks tensors and the \ke metrics.
\begin{corollary}\label{co:symm}
Let $A_{s \ol\beta\,\ol\delta}= g_{\alpha\ol\beta}A^\alpha_{s\ol\delta}$.
Then
\begin{equation}
A_{s \ol\beta\,\ol\delta}=A_{s \ol\delta\,\ol\beta}.
\end{equation}
\end{corollary}

Next, we introduce a {\it global} function $\varphi(z,s)$, which is by
definition the {\em pointwise inner product of the canonical lift $v$ of
$\pt/\pt s$ at $s\in S$} with itself with respect to $\omega_\cX$.
\begin{definition}\label{de:varphi}
\begin{equation}\label{eq:varphidef}
\varphi(z,s) := \langle \pt_s + a_s^\alpha \pt_\alpha, \pt_s + a_s^\beta
\pt_\beta  \rangle_{\omega_\cX}
\end{equation}
\end{definition}

Since $\omega_\cX$ is not known to be positive definite in all directions,
$\varphi\geq 0$ is not known at this point.
\begin{lemma}\label{le:varphi_0}
\begin{equation}\label{eq:varphi}
\varphi =  g_{s\ol s} - g_{\alpha\ol s}
g_{s\ol\beta} g{^{\ol\beta\alpha}}
\end{equation}
\end{lemma}
\begin{proof}
The proof follows from Lemma~\ref{le:canlift} and
$$
\varphi = g_{s\ol s} + g_{s\ol\beta}a^\ol\beta_{\ol s} + a_s^\alpha g_{\alpha\ol s}
+ a_s^\alpha a_{\ol s}^\ol\beta \gab.
$$
\end{proof}

Denote by $\omega^{n+1}_\cX$ the $(n+1)$-fold exterior product, divided by
$(n+1)!$ and by $dV$ the Euclidean volume element in fiber direction. Then
the global real function $\varphi$ satisfies the following property:
\begin{lemma}\label{le:varphi}
$$
\omega^{n+1}_\cX= \varphi \cdot g \cdot dV\ii ds\wedge \ol{ds}.
$$
\end{lemma}

\begin{proof}
Compute the following $(n+1)\times(n+1)$-determinant
$$ \det
\left(
\begin{array}{cc}
g_{s\ol s} & g_{s\ol\beta}\\ g_{\alpha\ol s}& \gab
\end{array}
\right),
$$
where $\alpha,\beta=1,\ldots,n$.
\end{proof}

So far we are looking at {\it local} computations, which essentially only
involve derivatives of certain tensors. The only {\it global ingredient}\/
is the fact that we are given global solutions of the \ke equation.

The key quantity is the differentiable function $\varphi$ on $\cX$.
Restricted to any fiber it ties together the yet to be proven positivity
of the hermitian metric on the relative canonical bundle and the canonical
lift of tangent vectors, which is related to the harmonic \ks forms.

We use the Laplacian operators $\Box_{g,s}$ with non-negative eigenvalues
on the fibers $\cX_s$ so that for a real valued function $\chi$ the
Laplacian equals $\Box_{g,s}\chi
 = - g^{\ol\beta\alpha}\chi_{;\alpha\ol\beta}$.
\begin{proposition}\label{pr:elleq}
The following elliptic equation holds fiberwise:
\begin{equation}\label{eq:phiA}
(\Box_{g,s} + {\rm id})\varphi(z,s) = \|A_s(z,s)\|^2,
\end{equation}
where
$$
A_s=A^\alpha_{s\ol\beta} \frac{\pt}{\pt z^\alpha}dz^\ol\beta
$$
is the harmonic representative of the \ks class $\rho_s(\frac{\pt}{\pt
s})$ as above.
\end{proposition}
\begin{proof}
The essence to prove an elliptic equation for $\varphi$ involving tensors
on the fibers is to eliminate the second order derivatives with respect to
the base parameter. This is achieved by the left hand side of
\eqref{eq:phiA}. First,
\begin{eqnarray*}
g^{\ol\delta\gamma}g_{s\ol s;\gamma\ol\delta}
&=&g^{\ol\delta\gamma}\partial_s\partial_\ol s g_{\gamma\ol\delta}\\
&=&\partial_s(g^{\ol\delta\gamma}\partial_\ol s g_{\gamma\ol\delta})
 -a_s^{\gamma;\ol\delta}\partial_\ol s g_{\gamma\ol\delta}\\
&=&\partial_s\partial_\ol s \log g
 +a_s^{\gamma;\ol\delta} a_{\ol s\gamma;\ol\delta}\\
&=& g_{s \ol s}
 +a_s^\sigma{}_{;\gamma} a_{\ol s\sigma;\ol\delta} g^{\ol\delta\gamma}.
\end{eqnarray*}
Next
\begin{eqnarray*}
(a_s^\sigma a_{\ol s\sigma})_{;\gamma\ol\delta}g^{\ol\delta\gamma}
&=\left(a_s^\sigma{}_{;\gamma\ol\delta} a_{\ol s\sigma}
        +A_{s\ol\delta}^\sigma A_{\ol s\sigma\gamma}
        +a_{s;\gamma}^\sigma a_{\ol s\sigma;\ol\delta}
        +a_s^\sigma A_{\ol s\sigma\gamma;\ol\delta}
\right)g^{\ol\delta\gamma} .
\end{eqnarray*}
The last term vanishes because of the harmonicity of $A_s$, and
\begin{eqnarray*}
a_{s;\gamma\ol\delta}^\sigma g^{\ol\delta\gamma}
&=&A_{s\ol\delta;\gamma}^\sigma g^{\ol\delta\gamma}
  +a_s^\lambda R^\sigma{}_{\lambda\gamma\ol\delta}g^{\ol\delta\gamma}\\
&=&0-a_s^\lambda R^\sigma{}_\lambda\\
 &=& a_s^\sigma .
\end{eqnarray*}
\end{proof}

\begin{definition}\label{de:wpherm}
The \wp hermitian product on $T_sS$ is given by the $L^2$-inner product of
harmonic \ks forms:
\begin{equation}\label{eq:wpherm}
\Big\|\frac{\pt}{\pt s}\Big\|^2_{WP}:= \int_{\cX_s} A^\alpha_{s\ol\beta} A^\ol\delta_{\ol s\gamma }
g_{\alpha\ol\delta}g^{\ol\beta\gamma} g \, dV =
\int_{\cX_s} A^\alpha_{s\ol\beta} A^\ol\beta_{\ol s\alpha } g \, dV
\end{equation}
If $\frac{\pt}{\pt s^i}\in T_sS$ are part of a basis, we denote by
$G^{WP}_{i\ol\jmath}(s)$ the inner product, and set
$$
\omega^{WP}:= \ii G^{WP}_{i\ol\jmath} ds^i\we ds^{\ol\jmath}
$$
\end{definition}
Observe that the generalized \wp form is equal to a fiber integral:
\begin{proposition}[cf.\ \cite{f-s:extremal}]
\begin{equation}\label{eq:wpfib}
\omega^{WP} = \int_{\cX/S} \omega^{n+1}_\cX.
\end{equation}
\end{proposition}
The proposition implies the \ka property of $\omega^{WP}$ immediately. The
{\it proof} follows from Lemma~\ref{le:varphi} and
Proposition~\ref{pr:elleq}.

\section{Curvature of $R^{n-p}f_*\Omega^p_{\cX/S}(\cK_{\cX/S}^{\otimes m})$
-- Statement of the theorem and Applications}

\subsection{Statement of the theorem}\label{ss:curv} We consider an effectively
parameterized family $\cX \to S$ of canonically polarized manifolds,
equipped with \ke metrics of constant Ricci curvature $-1$. For any $m>0$
the direct image sheaves
$f_*\cK_{\cX/S}^{\otimes(m+1)}=f_*\Omega^n_{\cX/S}(\cK_{\cX/S}^{\otimes
m})$ are locally free. For values of $p$ other than $n$ we assume local
freeness of
$$
R^{n-p}f_*\Omega^p_{\cX/S}(\cK_{\cX/S}^{\otimes m}).
$$
The assumptions of Section~\ref{ss:dolb} are satisfied by Kodaira-Nakano
vanishing so that we can apply Lemma~\ref{le:dolb}. If necessary, we
replace $S$ by a (Stein) open subset, such that the direct image is
actually free, and denote by $\{\psi^1,\ldots,\psi^r\}\subset
R^{n-p}f_*\Omega^p_{\cX/S}(\cK_{\cX/S}^{\otimes m})(S)$ a basis of the
corresponding free $\cO_S$-module, and at a given point $s\in S$ we denote
by $\{(\pt/\pt s_i)|_s; i=1,\ldots,M\}$ a basis of the complex tangent
space $T_sS$ of $S$ over $\C$, where the $s_i$ are holomorphic coordinate
functions of a minimal smooth ambient space $U\subset \C^M$.

Let $A_{i \ol\beta}^\alpha(z,s) \pt_\alpha dz^{\ol\beta}$ be a harmonic
\ks form. Then for  $s \in S$ the cup product together with the
contraction defines
\begin{footnotesize}
\begin{eqnarray}
A_{i \ol\beta}^\alpha \pt_\alpha dz^{\ol\beta}\cup \textvisiblespace :
\cA^{0,n-p}(\cX_s,\Omega^p_{\cX_s}(\cK_{\cX_s}^{\otimes m})) &\to&
\cA^{0,n-p+1}(\cX_s,\Omega^{p-1}_{\cX_s}(\cK_{\cX_s}^{\otimes m}))\label{eq:cup1}\\ A_{\ol \jmath
\alpha}^\ol\beta \pt_\ol\beta dz^{\alpha}\cup \textvisiblespace :
\cA^{0,n-p}(\cX_s,\Omega^p_{\cX_s}(\cK_{\cX_s}^{\otimes m})) &\to&
\cA^{0,n-p-1}(\cX_s,\Omega^{p+1}_{\cX_s}(\cK_{\cX_s}^{\otimes m})).\label{eq:cup2}
\end{eqnarray}
\end{footnotesize}
\enlargethispage{.5cm} We will apply the above product to harmonic
$(0,n-p)$-forms. In general the result is not harmonic. We use the
notation $\psi^{\ol\ell}:= \ol{\psi^\ell}$ for sections $\psi_k$ (and a
notation of similar type for tensors on the fibers):

\begin{theorem}\label{th:curvgen}
The curvature tensor for $R^{n-p}f_*\Omega^p_{\cX/S}(\cK_{\cX/S}^{\otimes
m})$ is given by
\begin{eqnarray}\label{eq:curvgen}
R_{i\ol\jmath}^{\phantom{{i\ol\jmath}}\ol\ell k}(s)&=& m \int_{\cX_s}
\left( \Box + 1 \right)^{-1}(A_i\cdot A_\ol\jmath) \cdot(\psi^k \cdot
\psi^\ol\ell) g\/ dV\nonumber\\
&& \quad + m \int_{\cX_s} \left( \Box + m \right)^{-1} (A_i\cup\psi^k)
\cdot (A_\ol\jmath \cup \psi^\ol\ell) g\/ dV \\
&& \quad + m \int_{\cX_s} \left( \Box - m \right)^{-1}
(A_i\cup\psi^\ol\ell)\cdot (A_\ol\jmath \cup \psi^k) g\/ dV.
\nonumber
\end{eqnarray}
The only contribution in \eqref{eq:curvgen}, which may be negative,
originates from the harmonic parts in the third term. It equals
$$
- \int_{\cX_s} H(A_i\cup\psi^\ol\ell) \ol{H(A_j\cup\psi^\ol k)} g dV.
$$
\end{theorem}
Concerning the third term, the theorem contains the fact that the positive
eigenvalues of the Laplacian are larger than $m$.

Now the pointwise estimate \eqref{eq:hker2} of the resolvent kernel (cf.\
also Proposition~\ref{pr:resol}) translates into an estimate.
\begin{proposition}\label{pr:est1}
Let $f:\cX \to S$ be a family of canonically polarized manifolds, and $s
\in S$. Let a tangent vector of $S$ at $s$ be given by a harmonic \ks form
$A$ and let $\psi$ be a harmonic $(p,n-p)$-form on $\cX_s$ with values in
the $m$-canonical bundle. Then
\begin{equation}\label{eq:est0}
R(A,\ol A, \psi,\ol\psi) \geq P_n(d(\cX_s)) \cdot \|A\|^2\cdot
\|\psi\|^2 - \|H(A\cup \ol\psi)\|^2.
\end{equation}
\end{proposition}

For $p=n$ we obtain the following result.
\begin{corollary}\label{co:curvmcan}
For $f_*\cK_{\cX/S}^{\otimes(m+1)}$ the curvature equals
\begin{eqnarray}\label{eq:curvmcan}
R_{i\ol\jmath}^{\phantom{i\ol\jmath}\ol\ell k}(s)&=& m \int_{\cX_s}
\left( \Box + m \right)^{-1} (A_i\cup\psi^k)
\cdot (A_\ol\jmath \cup \psi^\ol\ell) g\/ dV\nonumber\\ && \quad +
m \int_{\cX_s} \left( \Box + 1 \right)^{-1}(A_i\cdot A_\ol\jmath)
\cdot(\psi^k \cdot \psi^\ol\ell)
 g\/ dV.
\end{eqnarray}
\end{corollary}
The first term in \eqref{eq:curvmcan} yields immediately Nakano
semipositivity, since the operator $(\Box +m)^{-1}$ is positive on the
respective tensors. In fact more can be shown for the curvature of the
direct image of relative $m$-canonical bundles.

Let
\begin{equation}\label{eq:inpro}
H^{\ol\ell k} = \int_{\cX_s} \psi^k \cdot \psi^\ol\ell g\, dV.
\end{equation}
\begin{corollary}\label{co:curv1}
Let $s\in S$ be any point. Let $\xi^i_k\in \C$. Then
\begin{equation}
R_{i\ol\jmath}^{\phantom{i\ol\jmath}\ol\ell k}(s) \xi^i_k\ol{\xi^j_\ell}
\geq m \cdot P_n(d(\cX_s) )\cdot G_{i\ol\jmath}^{WP}\cdot H^{\ol\ell k}
\cdot \xi^i_k\ol{\xi^j_\ell}.
\end{equation}
In particular the curvature is strictly Nakano-positive with the above
estimate.
\end{corollary}

Next, we set $m=1$ and take a dual basis $\{\nu_i\}\subset R^p
f_*\Lambda^p\cT_{\cX/S}(S)$ of the $\{\psi^k\}$ and normal coordinates at
a given point $s_0\in S$. Observing  that the role of conjugate and
non-conjugate tensors is being switched, we compute the curvature as
follows.
\begin{theorem}\label{th:curvgendual}
The curvature of $R^pf_*\Lambda^p\cT_{\cX/S}$ equals
\begin{eqnarray}\label{eq:curvgendual}
R_{i\ol\jmath   k \ol\ell }(s)&=&- \int_{\cX_s}
\left( \Box + 1 \right)^{-1}(A_i\cdot A_\ol\jmath)
\cdot(\nu_k \cdot \nu_\ol\ell) g\/ dV\nonumber\\
&& \quad - \int_{\cX_s} \left( \Box + 1 \right)^{-1} (A_i\wedge\nu_\ol\ell)
\cdot (A_\ol\jmath \wedge \nu_k) g\/ dV \\
&& \quad -  \int_{\cX_s} \left( \Box - 1 \right)^{-1}
(A_i\wedge \nu_k)\cdot (A_\ol\jmath \wedge \nu_\ol\ell) g\/ dV.
\nonumber
\end{eqnarray}
The only possible positive contribution arises from
$$
\int_{\cX_s}H(A_i\wedge \nu_k) H(A_\ol\jmath \wedge \nu_\ol\ell) g\/ dV.
$$
\end{theorem}
We observe that for $n=1$ the third term in \eqref{eq:curvgendual} is not
present and we have the formula for the classical \wp metric on \tei
space: It is known from the results of Wolpert that the classical \wp
metric for families of Riemann surfaces of genus larger than one has
negative curvature: According to \cite{wo} the sectional curvature is
negative, and the holomorphic sectional curvature is bounded from above by
a negative constant. A stronger curvature property, which is related to
strong rigidity, was shown in \cite{sch:teich}. The strongest result on
curvature by Liu, Sun, and Yau now follows immediately from
Corollary~\ref{co:curvmcan}:
\begin{corollary}[\cite{yau}]
The \wp metric on the \tei space of Riemann surfaces of genus $p>1$ is
dual Nakano negative.
\end{corollary}
\begin{proof}
Observe that for a universal family $f:\cX\to S$ the classical \wp metric
on $R^1f_* \cT_{\cX/S}$ corresponds to the $L^2$ metric on its dual bundle
$f_*(\cK_{\cX/S}^{\otimes2})$, which is Nakano positive according to
Corollary~\ref{co:curv1}.
\end{proof}

For $p=1$ we obtain the curvature for the generalized \wp metric from
\cite{sch:curv}, (cf.\ \cite{siu:canlift}). Again we can estimate the
curvature like in Proposition~\ref{pr:est1}.

The following case is of particular interest.

\begin{proposition}\label{pr:est2}
Let $f:\cX \to S$ be a family of canonically polarized manifolds and $s
\in S$. Let a tangent vectors of $S$ at $s$ be given by harmonic \ks forms
$A, A_1,\ldots, A_p$  on $\cX_s$. Let $R$ denote the curvature tensor for
$R^pf_*\Lambda^p\cT_{\cX/S}$. Then we have in terms of the \wp norms:
\begin{gather}\label{eq:est1}
R(A,\ol A, H(A_1\wedge \ldots \wedge A_p) ,\ol{H(A_1\wedge \ldots \wedge
A_p)}) \leq  \nonumber \hspace{10cm} \\ \qquad - P_n(d(\cX_s)) \cdot
\|A\|^2\cdot \| H(A_1\wedge \ldots \wedge A_p)  \|^2 + \|H(A\wedge
A_1\wedge \ldots \wedge A_p)\|^2.
\end{gather}
\end{proposition}
\begin{proof}
Since the $A$ and $A_i$ are $\ol\pt$-closed forms, we have $H(A\wedge
H(A_1\wedge \ldots\wedge A_p))=H(A\wedge A_1 \wedge \ldots \wedge A_p)$.
\end{proof}

Next, we define  higher \ks maps defined on the symmetric powers of the
tangent bundle of the base. For $p>0$ we let the morphism
\begin{equation}\label{eq:rhop}
\rho^p: S^p \cT_{S} \to R^pf_*\Lambda^p \cT_{\cX/S}
\end{equation}
send a symmetric power
$$
\frac{\pt}{\pt s^{i_1}}\otimes \ldots \otimes \frac{\pt}{\pt s^{i_p}}
$$
to the class of
$$
A_{i_1}\wedge \ldots \wedge A_{i_p}:=
A_{i_1\ol\beta_1}^{\alpha_1}\pt_{\alpha_1}dz^{\ol\beta_1}
\wedge \ldots \wedge A_{i_p\ol\beta_p}^{\alpha_p}\pt_{\alpha_p}dz^{\ol\beta_p}.
$$

\begin{definition}
Let the tangent vector $\pt/\pt s$ correspond to the harmonic \ks tensor
$A_s$. The generalized \wp function of degree $p$ on the tangent space is
\begin{gather*}
\| \pt/\pt s \|^{WP}_p  := \|A_s\|_p := \| \underbrace{A_s \wedge \ldots
\wedge A_s}_p \|^{1/p}\hspace{7cm} \\
:=  \left( \int_{\cX_s} H(A_s \wedge \ldots \wedge A_s)\cdot \ol{H(A_s
\wedge \ldots \wedge  A_s)} g \, dV \right) ^{1/2p}
\end{gather*}
\end{definition}

For the computation of the curvature, we assumed that the sheaves
$R^p\Lambda^p\cT_{\cX/S}$ are locally free.

Observe that the coefficient $P_n(d(\cX_s))$ in \eqref{eq:est1} remains
bounded as long as the fibers are smooth independent of the local freeness
of the direct image sheaves.

Given a family over a curve, the $p$-th \wp function of tangent vectors
defines a hermitian (pseudo)metric on the curve, which we denote by $G_p$.
\begin{lemma}\label{le:curvgp}
The curvature $K_{G_p}$ of $G_p$, at points with $G_p(s)\neq 0$ satisfies
\begin{eqnarray}\label{eq:curvgp}
&&\\
K_{G_p} &\leq& \frac{1}{p}\left( - \frac{1}{c_{p,n}} P_n(d(\cX_s)) + \max_{A\neq
0}\{({\|A\|_{p+1}}/{\|A\|_{p}})^{2p+2}\}\right)
\text{\quad for \quad} p< n \nonumber \\
&&\\
K_{G_p} &\leq& -\frac{1}{p \cdot c_{p,n}} P_n(d(\cX_s)) \text{\quad for \quad} p= n \text{, or if }\/ G_{p+1}\equiv 0
\nonumber
\end{eqnarray}
for some $c_{p,n}>0$. These estimates are uniform on any relatively
compact subset of the moduli space.
\end{lemma}

\begin{proof}
Let $A^p$ be the {\em harmonic projection} of the $p$-fold exterior
product of $A_s$. Then the curvature tensor for $R^pf_*
\Lambda^p\cT_{\cX/C}$ satisfies
\begin{gather}
R(\pt_s,\pt_\ol s, A^p, \ol{A^p}) \geq -\frac{\pt^2}{\pt s\ol{\pt s}}
\log(G_p^p)\cdot (G_p^p)\cdot A^p\, \ol{A^p} =\\ \nonumber- p
\frac{\pt^2}{\pt s\ol{\pt s}} \log(G_p) \cdot \|A\|^{2p}_p = p G_p K_{G_p}
\|A\|^{2p}_p .
\end{gather}
Here $R(\pt_s,\pt_\ol s,\textvisiblespace,\textvisiblespace)$ is the
curvature form applied to the tangent vectors $\pt/\pt s$ and $\pt/\ol{\pt
s}$ resp. With respect to $G_p$, we identify $G_p=\|\pt/\pt
s\|^2_p=\|A_s\|^2_p$ so that
$$
R(A_s,\ol{A_s},A^p, \ol{A^p}) \geq p K_{G_p} \|A\|^{2p+2}_p.
$$
Now the estimate of Proposition~\ref{pr:est2} implies
\begin{gather*}
K_{G_p} \leq \frac{1}{p}\left( -P_n(d(\cX_s)) \|A\|_1^2 \|A\|_p^{2p} +
\|A\|_{p+1}^{2(p+1)} \right)\Big/\|A\|_p^{2(p+1)}.
\end{gather*}
The second term is not present for $p=n$.

Now the proof follows from Lemma~\ref{le:estwedge} below.
\end{proof}
\begin{lemma}\label{le:estwedge}
$$
\|A\|_p \leq c_{p,n}\|A\|_1.
$$
for some $c_{p,n}>0$.
\end{lemma}
\begin{proof}[Proof of Lemma~\ref{le:estwedge}]
Since $A_s$ is harmonic and $\| A_s\we\ldots\we A_s\|^2 \geq \|
H(A_s\we\ldots\we A_s)\|^2$, it is sufficient to show the pointwise
estimate (up to a constant that only depends upon the dimension and
degree)
$$
\|A_s \cdot A_{\ol s}\|^p(z) \gtrsim \|(A_s \we \ldots \we A_s) \cdot
\ol{(A_s \we \ldots \we A_s)}\|(z).
$$
We set
$$
B^\alpha_\gamma = A^\alpha_{s\ol\beta}(z) A^\ol\beta_{\ol s \gamma}(z)
$$
and use the symmetry $A_{a\ol\beta\ol\delta}$. It follows that
$$
\|(A_s \we
\ldots \we A_s) \cdot \ol{(A_s \we \ldots \we A_s)}\|(z)
$$
equals the determinant type expression
$$
\underline B(z)=\sum_{\sigma\in \mathfrak{S}_p} \epsilon(\sigma) B^{\alpha_1}_{\alpha_{\sigma(1)}}
\cdot\ldots\cdot B^{\alpha_p}_{\alpha_{\sigma(p)}}
$$
where the inner summations  take place with respect to the indices
$\alpha_1,\ldots, \alpha_p$.  At the given point $z$ we may assume that
$B^\alpha_\gamma$ is diagonal with non-negative entries
$\lambda_1,\ldots,\lambda_n$. Now it is easy to see that $\underline B(z)$
equals
$$
\sum_{1\leq i_1<  \ldots i_p \leq n} \lambda_{i_1} \cdot \ldots \cdot \lambda_{i_p}
$$
up to a numerical constant depending on $p$ and $n$. Again, up to a
numerical constant this can be estimated from below by
$$
\big(\sum^n_{i=1} \lambda_i\big)^p =  \|A_s \cdot A_\ol s\|^p(z) .
$$
\end{proof}
\begin{lemma}\label{le:estquot}
For any family over a base space $S$, which is mapped to a relatively
compact subset of the moduli space,
$$
\max_{A\neq 0}\{{\|A\|_{p+1}}/{\|A\|_{p}}\}
$$
is finite.
\end{lemma}
\begin{proof}
First we take a modification $\wt S \to S$ such that the pull backs of the
direct image sheaves $R^pf_*\Lambda^p\cT_{\cX/S}$ modulo torsion are
locally free. Then, by continuity the above quotients are bounded from
above. Next, we restrict the original family to the image of the locus of
torsion, i.e.\ to the support of the annihilator ideal, and repeat the
process.
\end{proof}

\subsection{Hyperbolicity conjecture of Shafarevich}
In \cite[3.2]{demasantacruz} Demailly gives a proof of the Ahlfors lemma
for hermitian metrics of negative curvature in the context of currents
using an approximation argument. Our argument depends upon the following
special case:

\begin{proposition}[Demailly]\label{pr:ahlschw}
Let $\gamma=\gamma(s) \ii ds\wedge \ol{ds}$, $\gamma(s)\geq 0$ be given on
an open disk $\Delta_R=\{|s|<R\}$, where $\log \gamma(s)$ is a subharmonic
function such that $(\ddb(\log \gamma) \geq A\, \gamma$ in the sense of
currents for some $A>0$. Let $\rho$ denote the Poincaré metric on
$\Delta_R$. Then $\gamma\leq \rho/A$ holds.
\end{proposition}

Before we construct a Finsler metric of negative holomorphic curvature, we
will treat families over curves directly.

Let $C$ be a compact smooth complex curve and $f: \cX\to C$ a holomorphic
family of smooth canonically polarized varieties. Let $0 \neq \pt/\pt s$
be a local coordinate vector field on $C$ and $A_s$ the corresponding
family of harmonic \ks forms. Then we set
$$
p_0 := \max \{p; H(\underbrace{A_s\wedge \ldots A_s}_p) \not\equiv 0 \text{ on some open }U \subset C \}.
$$
Observe that the zero set of $H(A_s \wedge \ldots \wedge A_s)$ is
analytic.

We do not have to assume that the direct images
$R^pf_*\Lambda^p\cT_{\cX/C}$ are locally free. This is the case over the
complement $C'=C\backslash \{c_1,\ldots,c_k\}$ of a finite set of points.
Near points $c_j$ we can compute \ks forms $B_s$, $s\in C$ in terms o a
differential trivialization so that the $L^2$-norms of the $B_s$ are
bounded near the critical points. Hence the norms of the harmonic
representatives $A_s$ are bounded near $c_j$. The same holds for the wedge
product of these. Because of the boundedness $\log G_{p_0}$ is subharmonic
on all of $C$.

\begin{proposition}\label{pr:nonisotr}
Let $f: \cX\to C$ a non-isotrivial holomorphic family of smooth
canonically polarized varieties over a curve, and let $0<p_0\leq n$ be
chosen as above. Then $\log G_{p_0}$ is subharmonic, and for the curvature
\begin{equation}\label{eq:curvgp1}
K_{G_{p_0}} \leq - \frac{1}{p \cdot c_{p,n}} P_n(d(\cX_s))
\end{equation}
holds so that the curvature current is negative.
\end{proposition}
As long as the fibers are smooth, the diameters $d(\cX_s)$ are bounded
from above. The \!\! {\em proof}\/ of the estimate follows from
Lemma~\ref{le:curvgp} at points, where the direct image of order $p_0$ is
locally free.

At points with {\em singular fibers}, $\log G_{p_0}$ need no longer be
subharmonic, unless current of integration is added. An elementary
(counter$\text{-}$)ex\-am\-ple in fiber dimension zero is the map from a
hyperelliptic to a rational curve.

By Proposition~\ref{pr:ahlschw}) the genus of the base must be larger than
one. This gives a short proof of the following version of Shafarevich's
hyperbolicity conjecture for canonically polarized varieties \cite{b-v,
keko,kk, kv1, kv2, m, v-z, v-z2}.

\smallskip

{\bf Application.} {\it If a compact curve $C$ parameterizes a
non-isotrivial family of canonically polarized manifolds, its genus must
be greater than one.}

\subsection{Finsler metric on the moduli stack} Different notions are
common. We do not assume the triangle inequality/convexity. Such metrics
are also called {\em pseudo-metrics} (cf.\ \cite{kobook}).
\begin{definition}\label{de:fins}
Let $Z$ be a reduced complex space and $TZ$ it Zariski tangent fiber
bundle. An upper semi-continuous function
$$
F:TZ \to [0,\infty)
$$
is called Finsler pseudo-metric (or pseudo-length function), if
$$
F(av)=|a|F(v) \text{ for all } a\in \C, v\in TZ.
$$
\end{definition}
The triangle inequality on the fibers not required for the definition of
the ''holomorphic'' (or ''sectional'') curvature.

All metrics $G_p$ from Section~\ref{ss:curv} are (upper semi-continuous)
Finsler pseudo-metrics.

A pseudometric $\gamma$ for a curve $C$ like in
Proposition~\ref{pr:ahlschw} and\ref{pr:nonisotr} may have isolated
zeroes.

We will use the fact that the holomorphic curvature of a Finsler metric at
a certain point $p$ in the direction of a tangent vector $v$ is the
supremum of the curvatures of the pull-back of the given Finsler metric to
a holomorphic disk through $p$ and tangent to $v$ (cf.\
\cite{abate-patrizio}). (For a hermitian metric, the holomorphic curvature
is known to be equal to the holomorphic sectional curvature.) In view of
Demailly's theorem, a Finsler metric may be defined in the above sense (as
long as the fibers are smooth, which is always the case. Furthermore, any
convex sum $G=\sum_j a_j G_j$, $a_j>0$ is upper semi-continuous and has
the property that $\log G$ restricted to a curve is subharmonic.

\begin{lemma}[cf.\ {\cite[Lemma~3]{sch:framas}}]\label{le:convsum}
Let $C$ be a complex curve and $G_j$ a collection of pseudo-metrics of
bounded curvature, whose sum has no common zero. Then the curvatures $K$
satisfy the following equation.
\begin{equation}\label{eq:curvest}
K_{\sum_{j=1}^k G_j} \leq \sum_{j=1}^k \frac{G_j^2}{(\sum_{i=1}^k G_i)^2} K_{G_j} .
\end{equation}
\end{lemma}

Now with Lemma~\ref{le:convsum} and Lemma~\ref{le:curvgp} we can construct
convex sums of the metrics $G_p$ with negative holomorphic curvature. In
this way we arrive at a (upper semi-continuous) Finsler metric rather than
a pseudo-metric. The convex sum accounts both cases where some $G_p$
vanishes or not. The metric is primarily given on local universal
families, but intrinsically given. It descends to the coarse moduli space
in the orbifold sense.

\begin{theorem}\label{pr:exfins}
On any relatively compact subset of the moduli space of canonically
polarized manifolds there exists a Finsler orbifold metric, i.e.\ a
Finsler metric on the moduli stack, whose holomorphic curvature is bounded
from above by a negative constant.
\end{theorem}

\section{Computation of the curvature}
We know from Lemma~\ref{le:dolb} that the metric tensor for
$R^{n-p}f_*\Omega^p_{\cX/S}(\cK_{\cX/S})$ on the base space $S$ is given
in terms of an integral which involves harmonic representatives of certain
cohomology classes and that these are the restrictions of certain
$\ol\pt$-closed differential forms on the total space. We already saw that
these give rise to fiber integrals. When we actually compute derivatives
with respect to the base, we will apply Lie derivatives with respect to
horizontal lifts of tangent vectors of the base. At this point we need to
take into account that the exterior derivatives $\pt$ has to be taken with
respect to the hermitian metric on the relative canonical line bundle.
Here covariant derivatives with respect to the total space occur (at least
in an implicit way). Since we are dealing with alternating forms we may
use covariant derivatives with respect to the \ka structure on the fibers,
which is necessary to somewhat simplify the computations. Again, we will
use the semi-colon notation for covariant derivatives and use a $|$-symbol
for ordinary derivatives, if necessary. Greek indices are being used for
fiber coordinates, Latin indices indicate the base direction. Dealing with
alternating forms, for instance of degree $(p,q)$, extra coefficients of
the form $1/p!q!$ are sometimes customary; these play a role, when the
coefficients of an alternating form are turned into skew-symmetric tensors
by taking the average. However, for the sake of a halfway simple notation,
we follow the better part of the literature and leave these to the reader.

\subsection{Setup}
As above, we denote by $f:\cX \to S$ a smooth family of canonically
polarized manifolds and we pick up the notation from
Section~\ref{se:posi}. The fiber coordinates were denoted by $z^\alpha$
and the coordinates of the base by $s^i$. We set $\pt_i=\pt/\pt s^i$,
$\pt_\alpha=\pt/\pt z^\alpha$.

Again we have {\em horizontal lifts of tangent vectors and coordinate
vector fields on the base}
$$
v_i= \pt_i + a_i^\alpha  \pt_\alpha.
$$
As above we have the corresponding harmonic representatives
$$
A_i=A^\alpha_{i\ol\beta}\pt_\alpha dz^{\ol \beta}
$$
of the \ks classes $\rho(\pt_i|_{s_0})$.

For the computation of the curvature it is sufficient to treat the case
where $\dim S =1$. We set $s=s_1$ and $v_s=v_1$ etc. In this case we write
$s$ and $\ol s$ for the indices $1$ and $\ol 1$ so that
$$
v_s= \pt_s + a_s^\alpha \pt_\alpha
$$
etc.

Sections of \RP will be denoted by letters like $\psi$.
\begin{gather*}
\psi|_{\cX_s} =
\psi_{\alpha_1,\ldots,\alpha_p,\ol\beta_{p+1},\ldots,\ol\beta_n}
dz^{\alpha_1}\wedge \ldots \wedge dz^{\alpha_p}\wedge
dz^{\ol\beta_{p+1}}\we
\ldots \we dz^{\ol\beta_n} \\
= \psi_{A_p\ol B_{n-p}} dz^{A_p}\we dz^{\ol B_{n-p}} \hspace{5cm}
\end{gather*}
where $A_p=(\alpha_1,\ldots,\alpha_p)$ and $\ol B_{n-p}=(\ol\beta_{p+1},
\ldots,\ol\beta_n)$. The further component of $\psi$ is
$$
\psi_{\alpha_1,\ldots,\alpha_p,\ol\beta_{p+1},\ldots,\ol\beta_{n-1},\ol s}
dz^{\alpha_1}\wedge \ldots \wedge dz^{\alpha_p} \we dz^{\ol\beta_{p+1}}\we
\ldots \we dz^{\ol\beta_{n-1}}\we \ol{ds}.
$$
Now Lemma~\ref{le:dolb} implies
\begin{equation}
\psi_{\alpha_1,\ldots,\alpha_p,\ol \beta_{p+1}, \ldots, \ol\beta_n |\ol s}
= \sum_{j=p+1}^n (-1)^{n-j}
\psi_{\alpha_1,\ldots,\alpha_p,\ol \beta_{p+1}, \ldots, \wh{\ol\beta}_j, \ldots, \ol\beta_n ,\ol s|\ol\beta_j }.
\end{equation}
Since these are the coefficients of alternating forms, on the right-hand
side, we may also take the covariant derivatives with respect to the given
structure on the fibers
$$
\psi_{\alpha_1,\ldots,\alpha_p,\ol \beta_{p+1},
\ldots, \wh{\ol\beta}_j, \ldots, \ol\beta_n ,\ol s;\ol\beta_j}.
$$
\subsection{Cup-Product}
We define the cup-product of a differential form with values in the
relative holomorphic tangent bundle and an (line bundle valued)
differential form now in terns of local coordinates.
\begin{definition}\label{de:cup}
Let
$$
\mu= \mu^\sigma_{\alpha_1,\ldots,\alpha_p,\ol\beta_1,\ldots, \ol\beta_q}\pt_\sigma \,
dz^{\alpha_1}\we\ldots\we dz^{\alpha_p}\we dz^{\ol\beta_1}\we\ldots\we dz^{\ol\beta_q},
$$
and
$$
\nu= \nu_{\gamma_1,\ldots,\gamma_a,\ol\delta_1,\ldots,\ol\delta_b}
dz^{\gamma_1}\we\ldots\we dz^{\gamma_a}\we dz^{\ol\delta_1}
\we\ldots\we dz^\ol{\delta_b}
$$
Then
\begin{gather}\label{eq:cup}
\mu\cup\nu := \mu^\sigma_{\alpha_1,\ldots,\alpha_p,\ol\beta_1,\ldots,
\ol\beta_q}
\nu_{\sigma\gamma_2,\ldots,\gamma_a,\ol\delta_1,\ldots,\ol\delta_b}
dz^{\alpha_1}\we\ldots\we dz^{\alpha_p}  \\
\nonumber \hspace{3cm}\we dz^{\ol\beta_1}\we\ldots\we dz^{\ol\beta_q}\we
dz^{\gamma_2}\we\ldots\we dz^{\gamma_q}\we dz^{\ol\delta_1} \we\ldots\we
dz^\ol{\delta_b}
\end{gather}
\end{definition}

\subsection{Lie derivatives}
Let again the base be smooth, $\dim S=1$ with local coordinate $s$. Then
the induced metric on \RP is given by  \eqref{eq:inpro}, where the
pointwise inner product equals
$$
\psi^k \cdot \psi^\ol\ell g\, dV = (\ii)^n (-1)^{n(n-p)} \frac{1}{g^m} \psi^k \we \psi^\ol\ell,
$$
and where $1/g^m$ stands for the hermitian metric on the $m$-th canonical
bundle on the fibers.
\begin{lemma}\label{le:Lieder}
$$
\frac{\pt}{\pt s} H^{\ol\ell k} = \int_{\cX_s}
L_v(\psi^k \cdot \psi^\ol\ell) g\, dV = \langle L_v \psi^k, \psi^\ell  \rangle + \langle  \psi^k, L_\ol v  \psi^\ell \rangle ,
$$
where $L_v$ denotes the Lie derivative with respect to the canonical lift
$v$ of the coordinate vector field $\pt/\pt s$.
\end{lemma}
\begin{proof}
Taking the Lie derivative is not type-preserving. We need the
$(1,1)$-component: $L_v(g_{\alpha\ol\beta})= \big[ \pt_s + a^\alpha_s
\pt_\alpha , g_{\alpha,\ol\beta} \big]_{\alpha\ol\beta} =
g_{\alpha\ol\beta|s} + a^\gamma_{s}g_{\alpha\ol\beta;\gamma}+
a^\gamma_{s;\alpha}g_{\gamma\ol\beta} =-
a_{s\ol\beta;\alpha}+a^\gamma_{s;\alpha}g_{\gamma\ol\beta}=0 $. So
$L_v(det(g_{\alpha\ol\beta}))=0$.
\end{proof}
\begin{equation}
L_v\psi = L_v\psi' + L_v\psi'',
\end{equation}
where $L_v\psi'$ is of type $(p,n-p)$ and $L_v\psi''$ is of type
$(p-1,n-p+1)$. We have
\begin{eqnarray}
L_v\psi' &=&   \big[\pt_s + a^\alpha_s \pt_\alpha, \psi_{A_p\ol B_{n-p}} dz^{A_p } dz^{\ol B_{n-q}}\big]_{(p,n-p)}
\nonumber   \\
&=& (\psi_{;s} + a^\alpha_s \psi_{;\alpha} + \sum_{j=1}^p a^\alpha_{s;\alpha_j}
\psi_{
{\tiny\vtop{
\hbox{$\alpha_1,\ldots,\alpha,\ldots,\alpha_p\ol B_{n-p}\;$}\vskip-.8mm
\hbox{$\phantom{\alpha_1,\ldots,}{|\atop j} $}}}}) dz^{A_p}\we dz^{\ol B_{n-p}} \label{eq:lvprime} \\
L_v\psi'' &=&\big[\pt_s + a^\alpha_s \pt_\alpha, \psi_{A_p\ol B_{n-p}} dz^{A_p } dz^{\ol B_{n-q}}\big]_{(p-1,n-p+1)}
\nonumber   \\
&=& \sum^p_{j=1} A^\alpha_{s\ol\beta_p}
\psi^k_{ {\tiny\vtop{ \hbox{$\alpha_1,\ldots,\alpha,\ldots,\alpha_p\ol
B_{n-p}$\;}\vskip-.8mm \hbox{$\phantom{\alpha_1,\ldots,}{|\atop j} $}}}} \nonumber \\
&& \quad
\vtop{\hbox{$dz^{\alpha_1}\we\ldots\we dz^{\ol\beta_p}\we\ldots\we
dz^{\alpha_p} \we dz^{\ol\beta_{p+1}}\we\ldots\we
dz^{\ol\beta_n}$}\hbox{$\phantom{dz^{\alpha_1}\we\ldots\we \; }{|\atop j} $}} \label{eq:lvsecond}
\end{eqnarray}
We also note the values for the derivatives with respect to $\ol v$.
\begin{eqnarray}
L_\ol v\psi' &=&   \big[\pt_\ol s + a^\ol\beta_\ol s \pt_\ol\beta, \psi_{A_p\ol B_{n-p}}
dz^{A_p } dz^{\ol B_{n-q}}\big]_{(p,n-p)} \nonumber   \\
&=& (\psi_{;\ol s} + a^\ol\beta_\ol s \psi_{;\ol\beta} + \sum_{j=1}^p a^\ol\beta_{\ol s;\ol\beta_j}
\psi_{
{\tiny\vtop{
\hbox{$A_p \ol\beta_{p+1},\ldots,\ol\beta,\ldots,{\ol\beta}_n\;$}\vskip-.8mm
\hbox{$\phantom{A_p \ol\beta_{p+1},\ldots,}{|\atop j}$}}}})
dz^{A_p}\we dz^{\ol B_{n-p}} \label{eq:lvbprime} \\
L_\ol v\psi'' &=& \big[\pt_\ol s + a^\ol\beta_\ol s \pt_\ol\beta, \psi_{A_p\ol B_{n-p}} dz^{A_p }
dz^{\ol B_{n-q}}\big]_{(p+1,n-p-1)} \nonumber   \\
&=& \sum^n_{j=p+1} A^\ol\beta_{\ol s \alpha_{p+1}}
\psi^k_{\tiny\vtop{ \hbox{$\alpha_1,\ldots,\alpha_p,\ol\beta_{p+1},\ldots,\ol \beta,\ldots,\ol\beta_n\;$}\vskip-.8mm
\hbox{$\phantom{\alpha_1,\ldots,\alpha_p,\ol\beta_{p+1},\ldots,}{|\atop j}$}}} \nonumber \\ && \quad
\vtop{\hbox{$dz^{\alpha_1}\we\ldots\we dz^{\alpha_p}\we dz^{\ol\beta_1}\we\ldots\we dz^{\alpha_{p+1}}\we\ldots
\we dz^{\ol\beta_n} $}
\hbox{$\phantom{dz^{\alpha_1}\we\ldots\we dz^{\alpha_p}\we dz^{\ol\beta_1}\we\ldots\we dz}{|\atop j}$}} \label{eq:lvbsecond}
\end{eqnarray}
\begin{lemma}
\begin{eqnarray}
  (L_v\psi^k)'' &=& A_s \cup \psi^k\label{eq:2} \\
  (L_\ol v\psi^k)'' &=& (-1)^p A_\ol s \cup \psi^k \label{eq:3}
\end{eqnarray}
\end{lemma}
\begin{proof}[Proof of \eqref{eq:2}.]
By \eqref{eq:lvsecond} we have
\begin{gather*}
L_v\psi'' = \hspace{10cm}\\ =\sum^p_{j=1}
A^\alpha_{s\ol\beta_p}\psi^k_{\alpha_1,\ldots,
\wh\alpha_j,\ldots,\alpha_p,\alpha, \ol B_{n-p}} dz^{\alpha_1}\we
\ldots\we \wh{dz^{\alpha_j}}\we\ldots \we dz^{\alpha_p} \we
dz^{\ol\beta_p}\we
\ldots\we dz^{\ol\beta_n}\\
= (-1)^{p-1}\sum^p_{j=1}
A^\alpha_{s\ol\beta_p}\psi^k_{\alpha,\alpha_1,\ldots,\alpha_{p-1},\ol\beta_{p},\ldots,\ol\beta_{n}}
dz^{\alpha_1}\we\ldots\we\ldots \we dz^{\alpha_{p-1}} \we
dz^{\ol\beta_p}\we \ldots\we dz^{\ol\beta_n}.
\end{gather*}
\end{proof}
\begin{proof}[Proof of \eqref{eq:3}]
The claim follows in a similar way from \eqref{eq:lvbsecond}.
\end{proof}

The situation is not quite symmetric because of Lemma~\ref{le:dolb}, which
implies that the contraction of the global $(0,n-p)$-form $\psi$ with
values in $\Omega^p_{\cX/S}(\cK_{\cX/S})$ is well-defined. Like in
Definition~\ref{de:cup} we have a cup-product on the total space
(restricted to the fibers).
\begin{eqnarray*}
\ol v \cup \psi &=& (\pt_\ol s + a^\ol\beta_\ol s \pt_\ol\beta)  \cup \psi \\
&=&\psi_{A_p,\ol s, \ol \beta_{p+1},\ldots,\ol\beta_{n-1}} +
a^\ol\beta_\ol s \psi_{A_p,\ol\beta,\ol\beta_{p+1},\ldots,\ol \beta_{n-1} }
\end{eqnarray*}
\begin{lemma}\label{le:lvpsi1}
\begin{equation}\label{eq:lvpsi1}
L_\ol v\psi'= (-1)^p\ol\pt (\ol v \cup \psi).
\end{equation}
\end{lemma}
\begin{proof}
The proof follows from the fact that, according to Lemma~\ref{le:dolb},
$\psi$ is given by a $\ol\pt$-closed $(0,n-p)$-form on the total space
$\cX$ with values in a certain holomorphic vector bundle.
\end{proof}

We will need that the forms $\psi$ on the fibers are {\em also harmonic
with respect to $\pt$} (which was defined as the connection of the line
bundle $\cK^{\otimes m}_{\cX/S}$). First, we note the following fact,
which immediately follows from the fact that the curvature of
$(\cK_{\cX/S},g^{-1})$ equals $-\omega_\cX$. We will need this fact for
both the total space and the restriction to fibers.
\begin{lemma}\label{le:ddb}
\begin{equation}\label{eq:ddb}
 \ii [\ol\pt, \pt] = - m L_\cX,
\end{equation}
where $L_\cX$ denotes the multiplication with $\omega_{\cX}$.
\end{lemma}
Now:
\begin{lemma}\label{le:boxdboxdb}
The following equation holds on  $\cA^{(p,q)}(\cK^{\otimes m}_{\cX_s})$.
\begin{equation}
\Box_\pt = \Box_\ol\pt + m\cdot  (n-p-q) \cdot id.
\end{equation}
In particular, the harmonic forms $\psi \in \cA^{(p,n-q)}(\cK_{\cX_s})$
are also harmonic with respect to $\pt$.
\end{lemma}
\begin{proof}
We use the formulas
$$
\ii\, \ol\pt^* = [\Lambda,\pt] \text{\quad and \quad} -\ii \pt^* = [\Lambda,\ol \pt],
$$
where $\Lambda$ denotes the adjoint operator to $L$. Then
$$
\Box_\pt-\Box_\ol\pt = [\Lambda,\ii (\pt \ol\pt + \ol\pt \pt)]=[\Lambda,m\cdot \omega_\cX] =m\cdot(n-p-q)\cdot id.
$$
\end{proof}
Now we compute the curvature in the following way. Because of
\eqref{eq:lvpsi1}
$$
\langle \psi^k, L_\ol v (\psi^\ell)'  \rangle = 0
$$
holds for all $s\in S$ so that by Lemma~\ref{le:Lieder}
\begin{gather*}
\frac{\pt}{\pt s} H^{\ol\ell k} = \langle L_v \psi^k, \psi^\ell  \rangle +
\langle  \psi^k,  L_\ol v \psi^\ell \rangle 
=\langle (L_v \psi^k)', \psi^\ell \rangle + \langle \psi^k, (L_\ol v
\psi^\ell)' \rangle \\
=\langle (L_v \psi^k)', \psi^\ell \rangle.
\end{gather*}

Later in the computation we will use normal coordinates (of the second
kind) at a given point $s_0\in S$. The condition $(\pt/\pt s)H^{\ol\ell
k}|_{s_0}=0$ for all $k,\ell$ means that for $s=s_0$ the harmonic
projection
\begin{equation}\label{eq:HLv}
H((L_v \psi^k)') =0
\end{equation}
vanishes for all $k$.

In order to compute the second order derivative of $H^{\ol\ell k}$ we
begin with
\begin{equation}\label{eq:dsH}
\frac{\pt}{\pt s} H^{\ol\ell k} = \langle L_v \psi^k, \psi^\ell \rangle.
\end{equation}
which contains both $(L_v \psi^k)'$ and $(L_v \psi^k)''$. Now
\begin{gather*}
\pt_\ol s\pt_s \langle \psi^k,\psi^\ell \rangle = \langle L_\ol v L_v
\psi^k , \psi^\ell \rangle +\langle L_v \psi^k,L_v\psi^\ell \rangle   \\ =
\langle L_{[\ol v,v]}\psi^k , \psi^\ell\rangle + \langle  L_vL_\ol
v\psi^k,\psi^\ell \rangle + \langle L_v \psi^k,L_v\psi^\ell \rangle \\ =
\langle L_{[\ol v,v]}\psi^k , \psi^\ell\rangle + \pt_s\langle L_\ol v
\psi^k, \psi^\ell\rangle -\langle L_\ol v\psi^k,L_\ol v\psi^\ell\rangle +
\langle L_v\psi^k,L_v\psi^\ell\rangle
\end{gather*}
We just saw that $\langle L_\ol v\psi^k , \psi^\ell \rangle \equiv 0$.
Hence for all $s\in S$
\begin{equation}\label{eq:Lolvv}
\pt_\ol s\pt_s \langle \psi^k,\psi^\ell \rangle = \langle L_{[\ol
v,v]}\psi^k , \psi^\ell\rangle -\langle L_\ol v\psi^k,L_\ol
v\psi^\ell\rangle + \langle L_v\psi^k,L_v\psi^\ell\rangle
\end{equation}
The fact that we are computing Lie-derivatives of $n$-forms (with values
in some line bundle) implies that
$$
\langle L_v\psi^k,L_v\psi^\ell\rangle =  \langle (L_v\psi^k)',(L_v\psi^\ell)'\rangle
-  \langle (L_v\psi^k)'',(L_v\psi^\ell)''\rangle,
$$
and
$$
\langle L_\ol v\psi^k,L_\ol v\psi^\ell\rangle =  \langle (L_\ol v\psi^k)',(L_\ol v\psi^\ell)'\rangle
-  \langle (L_\ol v\psi^k)'',(L_\ol v\psi^\ell)''\rangle.
$$
\begin{lemma}\label{le:vvb}
Restricted to the fibers $\cX_s$ the following equation holds for $L_{[\ol
v, v]}$ applied to $\cK^{\otimes m}_{\cX/S}$-valued functions and
differential forms resp.
    \begin{equation}\label{eq:vvb}
      L_{[\ol v, v]} =
      \big[-\varphi^{;\alpha}\pt_\alpha + \varphi^{;\ol\beta}\pt_{\ol\beta},\; \textvisiblespace \;\big]
      - m\cdot \varphi \cdot id
    \end{equation}
\end{lemma}
\begin{proof}
We first compute the vector field $[\ol v, v]$ on the fibers:
\begin{gather*}
[\ol v, v]= [\pt_\ol s + a^\ol\beta_\ol s \pt_\ol\beta, \pt_{s} +
a^\alpha_{s} \pt_\alpha ]\hspace{5cm} \\= \left(\pt_\ol s (a^\alpha_ s) +
a^\ol\beta_\ol s a^\alpha_{s|\ol\beta}\right)\pt_\alpha - \left( \pt_s
(a^\ol\beta_\ol s) + a^\alpha_{s}a^\ol\beta_{\ol
s|\alpha}\right)\pt_\ol\beta.
\end{gather*}
Now
\begin{gather*}
\pt_\ol s (a^\alpha_ s) = -\pt_\ol s (g^{\ol\beta\alpha}g_{s\ol\beta}) =
g^{\ol\beta\sigma} g_{\sigma\ol s| \ol \tau}g^{\ol\tau
\alpha}g_{s\ol\beta} - g^{\ol\beta\alpha}g_{s\ol \beta|\ol s} \\
\hspace{5cm} = g^{\ol\beta\sigma}a_{\ol s
\sigma;\ol\tau}g^{\ol\tau\alpha}a_{s\ol\beta} - g^{\ol\beta\alpha} g_{s\ol
s; \ol\beta}
\end{gather*}
Now \eqref{eq:varphi} implies that the coefficient of $\pt_\alpha$ is
$-\varphi^{;\alpha}$. In the same way the coefficient of $\pt_\ol \beta$
is computed.

Next, we compute the contribution of the connection on $\cK^{\otimes
m}_{\cX/S}$ which we denote by $[\ol v, v]_{\cK^{\otimes m}_{\cX/S}}$. We
use \eqref{eq:ddb}:
\begin{gather*}
[\pt_\ol s + a^\ol\beta_\ol s \pt_\ol\beta, \pt_{s} + a^\alpha_{s}
\pt_\alpha]_{\cK^{\otimes m}_{\cX/S}}\hspace{5cm}  \\ =-m\left(g_{s\ol s}
+ a^\ol\beta_{\ol s} g_{s\ol\beta} + a^\alpha_s g_{\alpha\ol s} +
a^\ol\beta_\ol s a^\alpha_s g_{\alpha\ol\beta} \right) = -m \varphi.
\end{gather*}
\end{proof}

\begin{lemma}\label{le:lvvpsi}
\begin{equation}\label{eq:lvvpsi}
\langle L_{[\ol v, v]} \psi^k , \psi^\ell  \rangle = -m \langle \varphi \psi^k, \psi^\ell\rangle
= - m \int_{\cX_s} (\Box + 1)^{-1}(A_s \cdot A_\ol s) \psi^k \psi^\ol\ell \, g \, dV
\end{equation}
\end{lemma}
\begin{proof}
The $\pt$-closedness of the $\psi^k$ can be read as
$$
\psi^k_{;\alpha} = \sum_{j=1}^p
\psi_{
{\tiny\vtop{
\hbox{$\alpha_1,\ldots,\alpha,\ldots,\alpha_p\ol B_{n-p};\alpha_j\;$}\vskip-.8mm
\hbox{$\phantom{\alpha_1,\ldots,}{|\atop j} $}}}}.
$$
Hence
\begin{eqnarray*}
[\varphi^{;\alpha}\pt_\alpha, \psi^k_{A_p\ol B_{n-p}}]'& = &
\varphi^{;\alpha}\psi_{;\alpha} +  \sum_{j=1}^p \varphi^{;\alpha}_{\;
;\alpha_j} \psi^k_{ {\tiny\vtop{
\hbox{$\alpha_1,\ldots,\alpha,\ldots,\alpha_p\ol B_{n-p}\;$}\vskip-.8mm
\hbox{$\phantom{\alpha_1,\ldots,}{|\atop j} $}}}}\\
& = & \sum_{j=1}^p \big( \varphi^{;\alpha} \psi^k_{ {\tiny\vtop{
\hbox{$\alpha_1,\ldots,\alpha,\ldots,\alpha_p\ol B_{n-p}\;$}\vskip-.8mm
\hbox{$\phantom{\alpha_1,\ldots,}{|\atop j} $}}}}\big)_{;\alpha_j}\\
&=& \pt\big( \varphi^{;\alpha}\pt _\alpha\cup \psi^k\big).
\end{eqnarray*}
Now
\begin{gather*}
\langle [\varphi^{;\alpha}\pt_\alpha, \psi^k_{A_p\ol
B_{n-p}}],\psi^\ell\rangle = \langle [\varphi^{;\alpha}\pt_\alpha,
\psi^k_{A_p\ol B_{n-p}}]',\psi^\ell\rangle \qquad \\ \qquad=\langle
\pt\big( \varphi^{;\alpha}\pt _\alpha\cup \psi^k\big),\psi^\ell \rangle =
\langle \varphi^{;\alpha}\pt _\alpha\cup \psi^k , \pt^* \psi^\ell \rangle
=0.
\end{gather*}
In the same way we get
$$
\langle [\varphi^{;\ol\beta}\pt_\ol\beta, \psi^k_{A_p\ol
B_{n-p}}],\psi^\ell\rangle =0,
$$
and, according to Lemma~\ref{le:vvb}, we are left with the desired term.
\end{proof}

\begin{proposition}\label{pr:baseq}
In view of \eqref{eq:cup1} and \eqref{eq:cup2} we have
\begin{eqnarray}
\ol\pt(L_v\psi^k)'&=&  \pt(A_s\cup \psi^k)\label{eq:0} \\
  \ol\pt^*(L_v\psi^k)'&=& 0 \label{eq:4} \\
  \pt^*(A_s\cup\psi^k) &=&0 \label{eq:5} \\
 \ol\pt^* (L_\ol v\psi^k)'&=&  \pt^* (A_\ol s \cup \psi^k) \label{eq:6}\\
  \ol\pt(L_\ol v\psi^k)'&=& 0 \label{eq:1} \\
  \ol\pt^*(A_\ol s \cup \psi^k) &=&0 \label{eq:7}
\end{eqnarray}
\end{proposition}
The proof of the above proposition is the technical part of this article
and will be given at the end of the manuscript.

{\it Proof of Theorem~\ref{th:curvgen}.} Again, we may set $i=j=s$ and use
normal coordinates at a given point $s_0\in S$.

We continue with \eqref{eq:Lolvv} and apply \eqref{eq:lvvpsi}. Let $G_\pt$
and $G_\ol\pt$ denote the Green's operators on the spaces of
differentiable $\cK_{\cX_s}$-valued $(p,q)$-forms on the fibers with
respect to $\Box_\pt$ and $\Box_\ol\pt$ resp. We know from
Lemma~\ref{le:boxdboxdb} that for $p+q=n$ the Green's operators $G_\pt$
and $G_\ol\pt$ coincide.

We compute $\langle (L_v\psi^k)', L_v\psi^\ell)'\rangle$: Since the
harmonic projection\\ $H ((L_v\psi^k)')=0$ vanishes for $s=s_0$, we have
\begin{gather*}
(L_v\psi^k)'= G_\ol\pt \Box_\ol\pt (L_v\psi^k)'= G_\ol\pt \ol\pt^*\ol\pt
(L_v\psi^k)' = \ol\pt^*G_\ol\pt \pt(A_s\cup\psi^k)
\end{gather*}
by \eqref{eq:4} and \eqref{eq:0}. The form $\ol\pt(L_v\psi^k)'=
\pt(A_s\cup \psi^k)$ is of type $(p,n-p+1)$ so that by
Lemma~\ref{le:boxdboxdb} on this space of such forms $G_\ol\pt=(\Box_\pt
+m)^{-1} $ holds.

Now
\begin{gather*}
\langle (L_v\psi^k)', (L_v\psi^\ell)'\rangle =\langle \ol\pt^*G_\ol\pt
\pt(A_s\cup\psi^k), (L_v\psi^\ell)' \rangle\\ = \langle G_\ol\pt
\pt(A_s\cup\psi^k), \pt (A_s \cup \psi^\ell) \rangle = \langle (\Box_\pt
+m)^{-1} \pt (A_s\cup \psi^k), \pt (A_s\cup \psi^\ell)\rangle\\
= \langle \pt^* (\Box_\pt +m)^{-1} \pt (A_s\cup \psi^k), A_s\cup
\psi^\ell\rangle.
\end{gather*}

Because of \eqref{eq:5}
\begin{gather*}
\langle (L_v\psi^k)', (L_v\psi^\ell)'\rangle = \langle (\Box_\pt
+m)^{-1}\Box_\pt (A_s \cup \psi^k) , A_s\cup \psi^\ell\rangle\\
= \langle A_s\cup \psi^k, A_s \cup\psi^\ell\rangle -m \langle (\Box+m
)^{-1}(A_s \cup \psi_k), A_s \cup \psi^\ell\rangle.
\end{gather*}
(For $(p-1, n-p+1)$-forms, we write $\Box=\Box_\pt=\Box_\ol\pt$.)
Altogether we have
\begin{equation}\label{eq:part2}
\langle L_v \psi^k , L_v \psi^\ell \rangle|_{s_0} = - m
\int_{\cX_s} (\Box +m)^{-1}(A_s\cup\psi^k)\cdot (A_\ol s\cup \psi^\ol\ell)\, g\, dV.
\end{equation}
Finally we need to compute $\langle L_\ol v\psi^k, L_\ol v \psi^\ell
\rangle$.

By equation \eqref{eq:3} we have that  $(\langle L_\ol v\psi^k)'', (L_\ol
v \psi^\ell)'' \rangle = \langle A_\ol s \cup \psi^k   , A_\ol s \cup
\psi^\ell \rangle$. Now Lemma~\ref{le:lvpsi1} implies that the harmonic
projections of the $(L_\ol v \psi^k)'$ vanish for all parameters $s$. So
\begin{gather*}
\langle (L_\ol v \psi^k)',  (L_\ol v \psi^\ell)'\rangle = \langle G_\ol\pt
\Box_\ol\pt (L_\ol v \psi^k)',  (L_\ol v \psi^\ell)'\rangle \\
\vtop{\hbox{$=$}\vskip-4mm\hbox{\tiny$\!\! \eqref{eq:1}$} } \langle
(G_\ol\pt \ol\pt^*\ol\pt L_\ol v \psi^k)', (L_\ol v \psi^\ell)'\rangle =
\langle
(G_\ol\pt \ol\pt L_\ol v \psi^k)', \ol\pt (L_\ol v \psi^\ell)'\rangle\\
\vtop{\hbox{$=$}\vskip-4mm\hbox{\tiny$\!\! \eqref{eq:6}$} } \langle
G_\ol\pt \pt^*(A_\ol s \cup \psi^k), \pt^*(A_\ol\ s\cup \psi^\ell)
\rangle. \hspace{3cm}
\end{gather*}
Now the $(p+1, n-p)$-form $\ol\pt^*(L_\ol v\psi^k)'=  \pt^* (A_\ol s \cup
\psi^k)$ is orthogonal to both the spaces of $\ol\pt$- and $\pt$-harmonic
forms. On these, we have by Lemma~\ref{le:boxdboxdb}
$$
\Box_\ol\pt=\Box_\pt - m\cdot id.
$$
We see that all eigenvalues of $\Box_\pt$ are larger or equal to $m$ for
$(p,n-p-1)$-forms.

{\bf Claim.} {\it Let $\sum_\nu \lambda_\nu \rho_\nu$ be the eigenfunction
decomposition of $A_\ol s\cup \psi^k$. Then all $\lambda_\nu > m$ or
$\lambda_0=0$. In particular $(\Box - m)^{-1}(A_\ol s\cup \psi^k)$
exists.}

In order to verify the claim, we consider $\pt^*(A_\ol s\cup \psi^k)=
\sum_\nu \pt^*(\rho_\nu)$ with
$$
\Box_\pt \pt^*(\rho_\nu) = \lambda_\nu \pt^*(\rho_\nu) = \Box_\ol\pt
\pt^*(\rho_\nu) + m \cdot\pt^*(\rho_\nu).
$$
This fact implies that $\sum_\nu \pt^*(\rho_\nu)$ is also the
eigenfunction expansion with respect to $\Box_\ol\pt$ and eigenvalues
$\lambda_\nu -m\geq 0$ of $\pt^*(A_\ol s\cup \psi^k)=\ol\pt^*(L_\ol v
\psi^k)$. The latter is orthogonal to the space of $\ol\pt$-harmonic
functions so that $\lambda_\nu-m=0$ does not occur. (The harmonic part of
$A_\ol v\cup \psi^k$ may be present though.) This shows the claim.

Now
$$
G_\ol\pt  \pt^*(A_\ol s \cup \psi^k) = (\Box_\pt -m)^{-1} \pt^*(A_\ol s \cup \psi^k)
$$
so that \eqref{eq:7} implies
\begin{gather*}
\langle (L_\ol v \psi^k)',  (L_\ol v \psi^\ell)'\rangle = \langle
(\Box_\pt -m)^{-1} \Box_\pt (A_\ol s \cup \psi^k) ,A_\ol s \cup \psi^\ell
\rangle \\
=  \langle A_\ol s \cup \psi^k ,A_\ol s \cup \psi^\ell \rangle + m\cdot
\langle (\Box_\pt -m)^{-1} (A_\ol s \cup \psi^k) ,A_\ol s \cup \psi^\ell
\rangle.
\end{gather*}
Now \eqref{eq:3} yields the final equation (again with $\Box_\ol\pt =
\Box_\pt = \Box$ for $(p+1,n-p-1)$-forms)
\begin{equation}\label{eq:part3}
\langle L_\ol v \psi^k,  L_\ol v \psi^\ell\rangle = m \int_{\cX_s} (\Box - m)^{-1}( A_\ol s \cup \psi^k)
\cdot (A_s \cup \psi^\ol\ell) \, g\, dV.
\end{equation}
The main theorem follows from \eqref{eq:lvvpsi}, \eqref{eq:part2},
\eqref{eq:Lolvv}, and \eqref{eq:part3}. \qed

\begin{proof}[Proof of Proposition~\ref{pr:baseq}]
We verify \eqref{eq:0}: We will need various identities. For simplicity,
we drop the superscript $k$. The tensors below are meant to be
coefficients of alternating forms on the fibers, i.e.\ skew-symmetrized.
\begin{equation}\label{eq:aux1}
\psi_{;s\ol\beta_{n+1}}= \psi_{;\ol\beta_{n+1}s} - m \cdot g_{s\ol\beta_{n+1}} \psi=
 m\cdot a_{s\ol\beta_{n+1}}
\end{equation}
\begin{gather} \label{eq:aux2}
\psi_{;\alpha\ol\beta_{n+1}}= \psi_{;\ol\beta_{n+1}\alpha}- m\cdot
g_{\alpha\ol\beta_{n+1}}\psi \hspace{5cm} \\ \nonumber - \sum^p_{j=1}
\psi_{\tiny \vtop{ \hbox{$\alpha_1,\ldots,\sigma,\ldots,\alpha_p, \ol
B_{n-p}$}\vskip-1.5mm\hbox{$ \phantom{\alpha_1,\ldots,}{|\atop j}$}}
}R^\sigma_{\; \alpha_j\alpha\ol\beta_{n+1}}-\sum^n_{j=p+1}
\psi_{\tiny\vtop{\hbox{$A_p \ol\beta_{p+1},
\ldots,\ol\tau,\ldots,\ol\beta_{n}  $}\vskip-1.5mm\hbox{$\phantom{A_p
\ol\beta_{p+1},
\ldots,}{|\atop j}  $}}} R^\ol\tau_{\:\ol\beta_j\alpha\ol\beta_n }\\
\nonumber = -m \cdot a_{s\ol\beta_{n+1}}\psi - \sum^p_{j=1}
a^\alpha_s\psi_{\alpha_1,\ldots,\sigma,\ldots,\alpha_p\ol
B_{n-p}}R^\sigma_{\; \alpha_j\alpha\ol\beta_{n+1}}
\end{gather}
\begin{gather}\label{eq:aux3}
a^\alpha_{s;\alpha_j\ol \beta_{n+1}} = A^\alpha_{s\ol\beta_{n+1};\alpha_j}
+ a^\sigma_s R^\alpha_{\; \sigma\alpha_j\ol\beta_{n+1}}
\end{gather}
Now, starting from \eqref{eq:lvprime} we get, using \eqref{eq:aux1},
\eqref{eq:aux2}, and \eqref{eq:aux3},
\begin{gather*}
\ol\pt L_v\psi' = \Big( \psi_{;s\ol\beta_{n+1}} +
A^\alpha_{s\ol\beta_{n+1}}\psi_{;\alpha} + a^\alpha_s
\psi_{;\alpha\ol\beta_{n+1}}  +
\sum^p_{j=1}a^\alpha_{s;\alpha_j\ol\beta_{n+1}}\psi_{\alpha_1,\ldots,\alpha,\ldots,\alpha_p,\ol
B_{n-p}}\\ \nonumber + \sum^p_{j=1} a^\alpha_{s;\alpha_j}
\psi_{\alpha_1,\ldots,\alpha,\ldots,\alpha_p,\ol B_{n-p};\ol\beta_{n+1}}
\Big) dz^{\ol\beta_{n+1}}\we dz^{A_p}\we dz^{\ol B_{n-p}}
\\ \nonumber 
= \Big(A^\alpha_{s\ol\beta_{n+1}}\psi_{;\alpha} + \sum^p_{j=1}
A^\alpha_{s\ol\beta_{n+1};\alpha_j}
\psi_{\alpha_1,\ldots,\alpha,\ldots,\alpha_p,\ol
B_{n-p}}\Big)dz^{\ol\beta_{n+1}}\we dz^{A_p}\we dz^{\ol B_{n-p}}
\end{gather*}
Because of the fiberwise $\pt$-closedness of $\psi$ this equals
\begin{gather*}
\sum^p_{j=1} \big(A^\alpha_{s\ol\beta_{n+1}} \psi_{\tiny \vtop{
\hbox{$\alpha_1,\ldots,\alpha,\ldots,\alpha_p, \ol
B_{n-p}$}\vskip-1.5mm\hbox{$ \phantom{\alpha_1,\ldots,}{|\atop j}$}}  }
\big)_{;\alpha_j}
dz^{\ol\beta_{n+1}}\we dz^{A_p}\we dz^{\ol B_{n-p}} \\
= (-1)^n \sum^p_{j=1} \big(A^\alpha_{s\ol\beta_{n+1}}
\psi_{\alpha,\alpha_2,\ldots,\alpha_p, \ol B_{n-p}} \big)_{;\alpha_1}
dz^{\alpha_1}\we dz^{A_{p-1}}\we dz^{\ol\beta_1}\we\ldots\we
dz^{\ol\beta_{n+1}}\\ \nonumber
=\pt \Big((-1)^n A^\alpha_{s\ol\beta_{n+1}}
\psi_{\alpha,\alpha_2,\ldots,\alpha_p,\ol\beta_{p+1},\ldots,\ol\beta_n}
dz^{A_{p-1}}\we dz^{\ol B_{n+1}}\Big) = \pt\big( A_s \cup \psi\big).
\end{gather*}
This shows \eqref{eq:0}.

Next, we prove \eqref{eq:4}. We begin with \eqref{eq:lvbprime}. We first
note
\begin{equation*}\label{eq:dsGamma}
\pt_s(\Gamma^\sigma_{\alpha\gamma})= -a^\sigma_{s; \alpha\gamma}
\end{equation*}
which follows in a straightforward way. Now this equation implies
\begin{gather*}
\pt_s(\psi_{;\gamma}) = \psi_{;s\gamma} - \sum^p_{j=1}
a^\sigma_{s;\alpha_j\gamma}\psi_{\tiny\vtop{\hbox{$\alpha_1,
\ldots,\sigma,\ldots,\alpha_p \ol B_{n-p}
$}\vskip-1.5mm\hbox{$\phantom{\alpha_1, \ldots,}{|\atop j} $} }}
\end{gather*}
so that (with $g^{\ol\beta_n\gamma}\psi_{;\gamma}=0$ and $\pt_s
g^{\ol\beta_n\gamma}= g^{\ol\beta_n\sigma} a^\gamma_{s;\sigma} $)
\begin{gather}\label{eq:aux01}
g^{\ol\beta_n\gamma}\psi_{;s\gamma} = -\psi_{;\gamma}
g^{\ol\beta_n\sigma} a^\gamma_{s;\sigma} +
\sum^p_{j=1}g^{\ol\beta_n\gamma}
a^\sigma_{s;\alpha_j\gamma}\psi_{\tiny\vtop{\hbox{$\alpha_1,
\ldots,\sigma,\ldots,\alpha_p \ol B_{n-p}
$}\vskip-1.5mm\hbox{$\phantom{\alpha_1, \ldots,}{|\atop j} $} }}
\end{gather}
follows. Next, since fiberwise $\psi$ is $\ol\pt^*$-closed,
\begin{gather}\label{eq:aux02}
g^{\ol\beta_n\gamma} (a^\alpha_s\psi_{;\alpha})_{;\gamma}=
g^{\ol\beta_n\gamma} a^\alpha_{s;\gamma}\psi_{;\alpha},
\end{gather}
and with the same argument
\begin{gather}\label{eq:aux03}
g^{\ol\beta_n\gamma} \big(\sum^p_{j=1}
a^\sigma_{s;\alpha_j}\psi_{\tiny\vtop{\hbox{$\alpha_1,
\ldots,\sigma,\ldots,\alpha_p \ol B_{n-p}
$}\vskip-1.5mm\hbox{$\phantom{\alpha_1, \ldots,}{|\atop j} $} }}
\big)_{;\gamma} = g^{\ol\beta_n\gamma} \sum^p_{j=1}
a^\sigma_{s;\alpha_j\gamma}\psi_{\tiny\vtop{\hbox{$\alpha_1,
\ldots,\sigma,\ldots,\alpha_p \ol B_{n-p}
$}\vskip-1.5mm\hbox{$\phantom{\alpha_1, \ldots,}{|\atop j} $} }}.
\end{gather}
Now $\ol\pt^*(L_v\psi')=0$ follows from \eqref{eq:aux01},
\eqref{eq:aux02}, and \eqref{eq:aux03}.

We come to the $\pt^*$-closedness \eqref{eq:5}  of $A_s\cup \psi$. We need
to show that
$$
\big(A^\alpha_{s\ol\beta_{n+1}}\psi_{\alpha,\alpha_2,\ldots,\alpha_p,\ol\beta_{p+1},\ldots,\beta_n}
\big)_{;\ol\delta}g^{\ol\delta\alpha_p}
$$
vanishes. Since $\pt^*\psi=0$ the above quantity equals
$$
A^\alpha_{s\ol\beta_{n+1};\ol\delta} \psi_{\alpha,\alpha_2,\ldots,\alpha_p,\ol B_{n-p}}g^{\ol\delta\alpha_p}.
$$
Because of the $\ol\pt$-closedness of $A_s$ this equals
$$
(A^{\alpha\alpha_p}_{s})_{;\ol \beta_{n+1}}\psi_{\alpha,\alpha_2,\ldots,\alpha_p,\ol B_{n-p}}g^{\ol\delta\alpha_p}.
$$
However,
$$
A^{\alpha\alpha_p}_s=A^{\alpha_p\alpha}_s
$$
whereas $\psi$ is skew-symmetric so that also this contribution vanishes.

The proof of \eqref{eq:6}, \eqref{eq:1}, and \eqref{eq:7} is similar, we
remark that \eqref{eq:1} follows from Lemma~\ref{le:lvpsi1}.
\end{proof}

\end{document}